\documentclass{amsart}
\usepackage{graphicx} 
\usepackage[utf8]{inputenc}
\usepackage{amsmath}
\usepackage{amsthm}
\usepackage{amsfonts}
\usepackage{amssymb}
\usepackage{cancel}
\usepackage[shortlabels]{enumitem}
\usepackage{mathrsfs}
\usepackage{tikz, tikz-cd}
\usepackage[colorlinks=true,linkcolor=blue,citecolor=blue,urlcolor=blue,citebordercolor={0 0 1},urlbordercolor={0 0 1},linkbordercolor={0 0 1}]{hyperref} 
\usepackage[nameinlink]{cleveref}

\theoremstyle{plain}
\newtheorem{thm}{Theorem}[section]
\newtheorem{cor}[thm]{Corollary}
\newtheorem{lem}[thm]{Lemma}

\newtheorem{prop}[thm]{Proposition}

\newtheorem*{thm*}{Theorem}

\newtheorem{thmx}{Theorem}

\theoremstyle{definition}
\newtheorem{rem}[thm]{Remark}

\newtheorem{defn}[thm]{Definition}

\newtheorem{ex}[thm]{Example}

\newcommand{\bbz}{\mathbb{Z}}

\newcommand{\XX}{\mathscr{X}}
\newcommand{\YY}{\mathscr{Y}}
\newcommand{\ZZ}{\mathscr{Z}}
\newcommand{\EE}{\mathscr{E}}
\newcommand{\FF}{\mathscr{F}}
\newcommand{\GG}{\mathscr{G}}

\newcommand{\PP}{\mathscr{P}}
\newcommand{\LL}{\mathscr{L}}

\newcommand{\Perf}{\text{Perf}}
\newcommand{\D}{\text{D}}
\newcommand{\Db}{\text{D}^b}
\newcommand{\Dq}{\text{D}_{\text{qc}}}

\newcommand{\Dbq}{\text{D}^b_{\text{qc}}}

\newcommand{\Dmc}{\text{D}^-_{\text{coh}}}
\newcommand{\Dbc}{\text{D}^b_{\text{coh}}}

\newcommand{\Dqch}{\text{D}_{\text{qc},\chi}}

\newcommand{\Dcch}{\text{D}_{\text{coh},\chi}}

\newcommand{\Dbcch}{\text{D}^b_{\text{coh},\chi}}

\newcommand{\Hom}{\text{Hom}}

\newcommand{\Ext}{\text{Ext}}

\newcommand{\Spec}{\text{Spec}}

\newcommand{\Gm}{\mathbb{G}_m}

\newcommand{\qc}{\text{qc}}

\newcommand{\Spl}{\mathcal{S}pl_{X/k}}

\newcommand{\spl}{\text{Spl}_{X/k}}

\newcommand{\Ssh}{\mathcal{SSH}_{X/k}}

\newcommand{\Sshx}{\mathcal{SSH}^{\xi}_{X/k}}

\newcommand{\ssh}{\text{SSH}_{X/k}}

\newcommand{\sshx}{\text{SSH}^{\xi}_{X/k}}

\newcommand{\sdx}{\text{sD}_{X/k}}

\newcommand{\Pic}{\text{Pic}}

\title{Semi-Homogeneous Sheaves and Twisted Derived Categories}
\author{Tyler Lane}
\begin{document}
\begin{abstract}
We produce twisted derived equivalences between torsors under abelian varieties and their moduli spaces of simple semi-homogeneous sheaves.  We also establish the natural converse to this result and show that a large class of twisted derived equivalences, including all derived equivalences, between torsors arise in this way.  As corollaries, we obtain partial extensions of the usual derived equivalence criterion for abelian varieties established by Orlov and Polishchuk.
\end{abstract}
\maketitle
\section*{Introduction}
With any derived equivalence $\Phi_{\EE}:\Db(Y) \xrightarrow{\sim} \Db(X)$ one can associate a morphism $i_{\EE}:Y \to \text{sD}_{X/k}$ from $Y$ to the moduli space of simple complexes on $X$ corresponding to the object $\EE \in \Db(Y \times X)$.  By \cite[Lemma 5.2]{LO15}, $i_{\EE}$ is an open immersion, so we may view $Y$ as a fine moduli space of complexes on $X$ and $\EE$ as the universal object on $Y \times X$.  Conversely, given a proper moduli space $M$ of complexes on $X$, one can ask whether or not the universal (twisted) object is the kernel of a (twisted) derived equivalence between $M$ and $X$.  In order for such an $M$ to be (twisted) derived equivalent to $X$ in this way, the objects parameterized by $M$ must satisfy certain strong conditions; for example, given any two objects $E_1, E_2 \in \Db(X)$ corresponding to distinct points of $M$, we must have
\begin{enumerate}
    \item $\text{dim}(\text{Ext}^i(E_1,E_1))=\binom{g}{i}$ for every integer $i \leq 1$, where $g=\text{dim}(X)$,
    \item $\text{dim}(\text{Ext}^i(E_1,E_2))=0$ for every $i \in \mathbb{Z}$.
\end{enumerate}

In general, it appears difficult to characterize these objects, especially those which satisfy the second condition.  However, when $X$ is a torsor under an abelian variety, these objects have been studied thoroughly in \cite{Mu78} and \cite{dJO22}: for example, it was shown in \cite{dJO22} that in characteristic $0$, any object $E \in \Db(X)$ satisfying $(1)$ is a \textit{simple semi-homogeneous sheaf} up to shift.  We show here that any two simple semi-homogeneous sheaves on $X$ with the same numerical invariants satisfy (1) and (2) and that the moduli space of such sheaves is a torsor under an abelian variety.  Now using the usual techniques, namely \Cref{A t: criterion}, we obtain the following.

\begin{thmx}[\Cref{2 t : equivalent to ssh}]\label{0 t: moduli}
Let $A$ be an abelian variety over a field $k$, and let $X$ be an $A$-torsor, let $M$ be a moduli space of simple semi-homogeneous sheaves on $X$.  Then there is a twisted derived equivalence
$$\Phi_{\EE}:\Db(M,\mu^{-1}) \to \Db(X),$$
where $\mu \in \text{Br}(M)$ is the universal obstruction and $\EE$ is the universal twisted sheaf on $M \times X$.
\end{thmx}

This result is an extension of work of Gulbrandsen and Mukai (see \cite{G08} and \cite{Mu98}).  When $k$ is algebraically closed and of characteristic $0$ (resp.~~ equal to $\mathbb{C}$), moduli spaces of simple semi-homogeneous vector bundles (resp.~ sheaves) were constructed in \cite{G08} (resp.~ \cite{Mu98}).  They also showed that when they exist, universal sheaves are the kernels of derived equivalences as in \Cref{0 t: moduli} (see \cite[Proposition 4.1]{G08} and \cite[Proposition 4.11]{Mu98}).  Our construction of the moduli spaces of simple semi-homogeneous sheaves is heavily inspired by those of \cite{G08} and \cite{Mu98}.

As an application of \Cref{0 t: moduli}, we can recover some known equivalences.  Line bundles serve as important examples of simple semi-homogeneous sheaves, and for any $\lambda \in \text{NS}_{X/k}(k)$, $\text{Pic}^\lambda_{X/k}$ is a moduli space of simple semi-homogeneous sheaves on $X$.  Using \Cref{0 t: moduli} one obtains twisted derived equivalences $\Db(\text{Pic}^\lambda_{X/k},\mu^{-1}) \simeq \Db(X)$.  Taking $X=A$, one recovers \cite[Theorem 4]{AAFH21}, and taking $\lambda=0$, one obtains the derived equivalences in \cite[Theorem 5]{RR23}.

Using \cite{dJO22}, we can obtain a partial converse to \Cref{0 t: moduli}. 
\begin{thmx}[\Cref{2 t: converse}]\label{0 t: converse}
Let $A$ be an abelian variety over a field $k$ and let $X$ be an $A$-torsor.  Let $Y$ be a smooth projective variety, and let $\Phi_{\EE}:D^b(Y,\mu^{-1}) \to \Db(X)$ be an equivalence. If either
\begin{enumerate}
    \item $k$ has characteristic $0$,
    \item $\mu \in Br_1(Y):=\text{Ker}(\text{Br}(Y) \to \text{Br}(Y_{\overline{k}}))$,
    \item $Y$ is a torsor under an abelian variety,
\end{enumerate}
then $Y$ is a moduli space of simple semi-homogeneous sheaves on $X$.  In particular, $Y$ is a torsor under an abelian variety.
\end{thmx}

We note that if \cite[Theorem 1.1]{dJO22} is true in arbitrary characteristic, then \Cref{0 t: converse} is true for arbitrary $Y$; the characteristic $0$ hypotheses are used exactly once in the proof of \cite[Theorem 1.1]{dJO22}, and it's not clear whether or not they're necessary (see \cite[Remark 1.3]{dJO22}).  In the proof \Cref{0 t: converse}, we work around these hypotheses to establish (2) and (3) in arbitrary characteristic.

Combining \Cref{0 t: moduli} and \Cref{0 t: converse} one obtains a criterion for certain twisted derived equivalences between torsors under abelian varieties.  To be precise, we've shown the following.

\begin{thmx}[\Cref{2 c: theorem C}] \label{0 c: criterion}
Let $A$ and $B$ be abelian varieties over field $k$, let $X$ be an $A$-torsor, and let $Y$ be a $B$-torsor.  Then there is a twisted derived equivalence $\Db(Y, \mu^{-1}) \simeq \Db(X)$ if and only if $Y$ is a moduli space of simple semi-homogeneous sheaves on $X$ and $\mu$ is the universal obstruction.  In particular, $\Db(Y) \simeq \Db(X)$ if and only if $Y$ is a fine moduli space of simple semi-homogeneous sheaves on $X$.
\end{thmx}

One might wonder to what degree \Cref{0 c: criterion} extends the existing derived equivalence criterion for abelian varieties over algebraically closed fields established by Orlov and Polishchuk (see \cite{Or02} and \cite{Pol96}).  Before explaining, we recall their main results.  In \cite{Or02}, it was shown that with any derived equivalence $\Phi_\EE:D^b(A) \xrightarrow{\sim} \Db(B)$ between two abelian varieties over an arbitrary field one can associate an isomorphism 
$$f_\EE=\begin{pmatrix}
    x & y \\
    z & w
\end{pmatrix}: A \times \hat{A} \rightarrow B \times \hat{B},$$
which is \textit{isometric}, i.e., such that
$$f_\EE^{-1}={\begin{pmatrix}
    \hat{w} & -\hat{y} \\
    -\hat{z} & \hat{x}
\end{pmatrix}}.$$
When $k$ is algebraically closed and of characteristic $0$, it was also shown that every isometric isomorphism arises this way.  In \cite{Pol96}, Polishchuk proved that when $k$ is algebraically closed of arbitrary characteristic, the existence of an isometric isomorphism $A \times \hat{A} \simeq B \times \hat{B}$ implies that $\Db(A) \simeq \Db(B)$.  Thus, two abelian varieties $A$ and $B$ over an algebraically closed fields are derived equivalent if and only if there exists an isometric isomorphism $A \times \hat{A} \simeq B \times \hat{B}$.  There are two obvious ways in which we can try to generalize this result.  We can either 
\begin{enumerate}
    \item Assume that $k$ is algebraically closed and replace the existence of an isometric isomorphism with a similar notion to obtain a criterion for twisted derived equivalence of the form $\Db(B,\beta^{-1}) \simeq \Db(A)$.
    \item Determine whether or not every isometric isomorphism between abelian varieties over an arbitrary field arises from a derived equivalence.
\end{enumerate}
Using some results from \cite{Pol12} we are able to obtain the natural twisted analogue of the usual criterion

\begin{thmx}[\Cref{3 p: Lagrangian equivalences}] \label{0 c: Lagrangians}
    Let $A$ and $B$ be abelian varieties over an algebraically closed field $k$.  Then there exists a twisted derived equivalence $\Db(B,\beta^{-1}) \simeq \Db(A)$ if and only if $B$ is isomorphic a Lagrangian abelian subvariety of $A \times \hat{A}$.
\end{thmx}
For those unfamiliar with symplectic biextensions and, in particular, with the notion of a Lagrangian abelian subvariety, we recall these notions in \Cref{3 SECTION}.  Twisted derived equivalences between $A$ and the Lagrangian abelian subvarieties of $A \times \hat{A}$ were established in \cite{Pol96} when $k$ is of characteristic $0$ (see \cite[Remarks 2.3.1]{Pol12} for the assumptions on the characteristic).  The methods we use here are quite different from those used in \cite{Pol96}, and it's not clear whether or not we obtain the same equivalences.

For abelian varieties over arbitrary fields, the situation is much more complicated.  For example, there exists an abelian variety $A$ and an isometric isomorphism from $A \times \hat{A}$ to itself which does not arise from any derived equivalence (see \Cref{4 e: counterexample} and \Cref{4 r: counterexample remark}).  Since $A$ is derived equivalent to itself, this doesn't show that the usual equivalence criterion fails over arbitrary fields.  We're also able to associate with any derived equivalence between an $A$-torsor and a $B$-torsor an isometric isomorphism between $A \times \hat{A}$ and $B \times \hat{B}$ in a way that generalizes \cite[Corollary 2.13.]{Or02}.  With this in mind, one might ask whether or not every isometric isomorphism between $A \times \hat{A}$ and $B \times \hat{B}$ arises from an equivalence between an $A$-torsor and a $B$-torsor via this construction.  Unfortunately, we're unable to determine whether or not this is true.  All we're able to show is the following.

\begin{thmx}[\Cref{3 p: theorem E}] \label{0 c: isometric iso}
    Let $A$ and $B$ be abelian vareities over a field $k$, and let $f:A \times \hat{A} \to B \times \hat{B}$ be an isometric isomorphism.  Then for any $A$-torsor $X$, there exists a $B$-torsor $Y$ and a twisted derived equivalence $\Db(Y,\mu^{-1}) \simeq \Db(X)$, where $\mu \in \text{Br}_1(Y)$.  In particular, if $k$ is a finite field, then $\Db(B) \simeq \Db(A)$.
\end{thmx}
Finally we mention that in \Cref{3 SECTION} we clarify a remark of Orlov's and show that any isometric isomorphism
$$f=\begin{pmatrix}
    x & y \\
    z & w
\end{pmatrix}: A \times \hat{A} \to B \times \hat{B}$$
can be factored as the composition of two isometric isomorphisms
$$f=\begin{pmatrix}
    x_2 & y_2 \\
    z_2 & w_2
\end{pmatrix} \circ \begin{pmatrix}
    x_1 & y_1 \\
    z_1 & w_1
\end{pmatrix}$$
such that $y_1$ and $y_2$ isogenies.  As was noted in \cite{MP17}, this fact is not obvious.  We show that this follows from a result of Polishchuk's (see \Cref{3 r: factorization}).

\subsection*{Outline}

In \Cref{1 SECTION} we recall some basic results on simple semi-homogeneous sheaves and show that the moduli space of such sheaves is a torsor under an abelian variety.  In \Cref{2 SECTION} we prove \Cref{0 t: moduli} and \Cref{0 t: converse}.  In \Cref{3 SECTION} we recall some generalities on symplectic biextensions, and in \Cref{4 SECTION} we prove \Cref{0 c: Lagrangians} and \Cref{0 c: isometric iso}.  In \Cref{A SECTION} we recall basic results on twisted derived categories.

\subsection*{Notation and Conventions}  Unless otherwise specified, we always work over an arbitrary field $k$ with algebraic closure $\overline{k}$.

Given a smooth, projective variety $X$, we let $\Db(X)$ denote the bounded derived category of the abelian category $\text{Coh}(X)$.  Given an element $\mu \in \text{Br}(X)$, we denote by $\Db(X,\mu)$ the bounded derived category of the abelian category of coherent $\mu$-twisted sheaves (see \Cref{A r: notation} and \Cref{A l: same}).  For a point $x \in X(k)$, we let $k(x)$ denote $i_{x,*}\mathcal{O}_{\text{Spec}(k)}$, where $i_{x}:\Spec(k) \to X$ denotes the inclusion of $X$.  When $k$ is algebraically closed we can define $k(x) \in \Db(X,\mu)$ similarly for any $x \in \Db(X,\mu)$ (see \Cref{A d: twisted skyscraper sheaves}).  Whether or not we're talking about the twisted version of $k(x)$ will be clear from context.

Given a morphism $X \to S$, an object $F \in \Db(X)$, and a morphism $T \to S$, we let $X_{T}:=X\times_{S}T$ and we denote by $F_{T}$ the derived pullback of $F$ along the induced map $X_T \to X$.  When $T=\text{Spec(R)}$ for a ring $R$, we write $X_{R}$ and $F_{R}$, and if $R=k(s)$ for some $s \in S$, we write $F_{s}$.  We use the same conventions for objects of twisted derived categories.

We let $\mathscr{D}_{X/k}$ denote the moduli stack of complexes on a smooth projective variety $X$ (see \cite[Section 5]{LO15} or \cite{Lie06}).  We call an object $P \in \mathscr{D}_{X/k}(T)$ \textit{simple} if the natural map $\Gm \to \text{Aut}_{\mathscr{D}_{X/k}}(P)$ is an isomorphism and denote by $\text{s}\mathscr{D}_{X/k}$ the open substack of $\mathscr{D}_{X/k}$ parameterizing simple complexes.  We let $\mathcal{S}pl_{X/k}$ denote the open substack of $\text{s}\mathscr{D}_{X/k}$ parameterizing complexes concentrated in degree $0$ and refer to it as the \textit{stack of simple sheaves}.  We call a $T$-point of $\mathcal{S}pl_{X/k}$ a $T$\textit{-flat family of simple sheaves on $X$}, and we call a $k$-point of $\mathcal{S}pl_{X/k}$ a \textit{simple sheaf}.  We denote by $\text{sD}_{X/k}$ (resp.~ $\text{Spl}_{X/k}$) the rigidification of $\text{s}\mathscr{D}_{X/k}$ (resp.~ $\mathcal{S}pl_{X/k}$) along $\Gm$ and refer to it as the \textit{moduli space of simple complexes} (resp.~ \textit{moduli space of simple sheaves}) \textit{on $X$}. 

All torsors are for the fppf topology.  Given an abelian variety $A$, an $A$-torsor $X$, and a morphism $f:T \to A$, we denote by $t_f:X_T \to X_T$ the restriction of the action map $X_T \times_T A_T  \to X_T$ to $X_T=X_T \times _T T$ via the map $\text{id}\times_T f$.
Given an abelian variety $A$, we let $\mathscr{P}_{A}$ denote the normalized Poincaré bundle on $A \times \hat{A}$.  Sometimes we simply write $\mathscr{P}$ when no confusion arises.  Given a morphism $h:T \to \hat{A}$, we let $\mathscr{P}_h$ denote the line bundle $(\text{id}\times h)^*\mathscr{P}$ on $A \times T$.  For a morphism $f:T \to A$, we denote by $\mathscr{P}_{\hat{f}}$ the line bundle $(f \times \text{id})^*\mathscr{P}$ on $T \times \hat{A}$. For a line bundle $L$ on $A$, we denote by $\phi_L:A \to \hat{A}$ the morphism induced by the line bundle $m^*L\otimes p_1^*L^{-1} \otimes p_2^*L^{-1}$ on $A \times A$, where $m:A \times A \to A$ is the multipication map.  We write either $[n]$ or $n$ for the multiplication by $n$ map.

Given a smooth, projective variety $X$, we denote by $\text{CH}_{\text{num}}(X)$ the Chow ring of $X$ modulo numerical equivalence.  We write $\text{CH}_{\text{num}}(X)_{\mathbb{Q}}$ for $\text{CH}_{\text{num}}(X) \otimes \mathbb{Q}$.  Let $\text{NS}_{X/k}:=\text{Pic}_{X/k}/\text{Pic}^0_{X/k,\text{red}}$ denote the Néron–Severi group scheme of $X$.   Let $\text{Br}(X)$ denote the Brauer group of $X$, and let $\text{Br}_1(X)=\text{Ker}(\text{Br}(X) \to \text{Br}(X_{\overline{k}}))$.  

\subsection*{Acknowledgments}  I would like to thank Brendan Hassett for suggesting what became this problem and for countless useful suggestions and conversations.  I would also like to thank Dan Abramovich, Joe Hlavinka, Eric Larson, Max Lieblich, Martin Olsson, and Mihnea Popa for useful conversations. 

\section{Simple Semi-Homogeneous Sheaves}\label{1 SECTION}
Throughout this section we fix a $g$-dimensional abelian variety $A$ over  a field $k$ and an $A$-torsor $X$.  Let $\text{sD}_{X/k}$ denote the moduli space of simple complexes on $X$ (see \cite[Section 5]{LO15}).  $A$ acts on $\sdx$ by pulling back along translation, and $\text{Pic}^0_{X/k}$ acts on $\sdx$ by the tensor product.  Letting $A$ act on $\sdx$ as above but letting $\text{Pic}^0_{X/k}$ act by inverting and then tensoring, we obtain an action of $A \times \text{Pic}^0_{X/k}$ on $\sdx$.  To be explicit, under this action, a point $a \in A(k)$ and a line bundle $L$ on $X$ send an object $F \in \Db(X)$ to $t_a^*F \otimes L^{-1}$.  Using the canonical isomorphism $\hat{A} \simeq \text{Pic}^0_{X/k}$ (see \cite[Theorem 3.0.3]{Al02}), we will always view this as an action of $A \times \hat{A}$ on $\text{sD}_{X}$.  Note that $A \times \hat{A}$ acts on $\spl$ by restriction.

For a simple object $F \in \Db(X)$, we denote by $\mathbf{S}(F)$ the stabilizer of $F$, viewed as a $k$-point of $\sdx$, for this action.  It is a closed subgroup of $A \times \hat{A}$.  When $F$ is a sheaf, this is the same as the stabilizer of $F$ for the action of $A \times \hat
A$ on $\spl$. 

\begin{rem}\label{1 r: MukaiNotation} In \cite{Mu78}, Mukai uses $\Phi(F)$ to denote what we call $\mathbf{S}(F)$. Note that by \cite[(3.5.1)]{Mu78}, $\Phi(F)$ and $\mathbf{S}(F)$ define the same subgroup of $A \times \hat{A}$.  Indeed, $\mathbf{S}(F)$ fits into a cartesian square
$$
\begin{tikzcd}
\mathbf{S}(F) \arrow[d] \arrow[r] & {\text{Spl}_{X/k}} \arrow[d, "\Delta"] \\
{A \times \hat{A}} \arrow[r]                      & {\text{Spl}_{X/k} \times \text{Spl}_{X/k}}          
\end{tikzcd}$$
where the bottom map sends a point $(h,\LL)\in A(T) \times \hat{A}(T)$ to $(t_h^*F_T \otimes \LL^{-1},F_T)$.  It follows from  \cite[\href{https://stacks.math.columbia.edu/tag/06QD}{Tag 06QD}]{Stacks}. that $\mathbf{S}(F)$ is the fppf sheafification of the presheaf which sends a scheme $T$ to the set of points $(h,\LL) \in A(T) \times \hat{A}(T)$ such that $t_{h}^*F_T \otimes \LL^{-1} \simeq F_T$.
\end{rem}

\begin{defn}\label{1 d: semi-homog}
  A simple sheaf on $X$ is $\textit{semi-homogeneous}$ if $\mathbf{S}(F)$ is $g$-dimensional.
\end{defn} 

Simple semi-homogeneous sheaves admit the following useful description:
\begin{prop}\label{1 p:ext}
 Let $F$ be a simple sheaf on $X$.  Then
$\text{dim}(\text{Ext}^1(F,F)) \geq g$.
Moreover, equality holds if and only if $F$ is semi-homogeneous.     
\end{prop}
Before giving the proof, we establish three useful lemmas.  The first one will be used frequently to reduce to working with abelian varieties over algebraically closed fields, and we will typically use it without mention.
\begin{lem}\label{1 l: reduction}
Let $\phi:Y \to X$ be an $A$-equivariant morphism between $A$-torsors.  The pullback map $\Spl \to \mathcal{S}pl_{Y/k}$ is an isomorphism, and the induced map $\spl \to \text{Spl}_{Y/k}$ is $A \times \hat{A}$-equivariant.  In particular, if $F$ is a simple sheaf on $X$ and $\phi:A \to X$ is a trivialization of $X$, then $\mathbf{S}(F)=\mathbf{S}(\phi^*F)$ so that $F$ is semi-homogeneous if and only if $\phi^*F$ is.
\end{lem}
\begin{proof}
    Any $A$-equivariant morphism between torsors is an isomorphism, so the pullback map is an isomorphism.  Moreover, the induced map on coarse moduli spaces is clearly equivariant with respect to the isomorphism $\text{id} \times \phi^*:A \times \text{Pic}_{X/k}^0 \to \text{A}\times \text{Pic}^0_{Y/k}$.  This map is, in fact, $A \times \hat{A}$-equivariant because the formation of the isomorphisms $\hat{A} \simeq \text{Pic}^0_{X/k}$ and $\hat{A} \simeq \text{Pic}^0_{Y/k}$ commutes with $\phi^*$.
\end{proof}

\begin{lem} \label{1 l: tensor S(F)}
    Let $A$ be an abelian variety over an algebraically closed field $k$.  Let $F \in \Db(A)$ be a simple object, and let $L$ be a line bundle on $A$.  Then the isomorphism
    $$f_L:\begin{pmatrix}
    1 & 0 \\
    \phi_L & 1
    \end{pmatrix}:A \times \hat{A} \to A \times \hat{A}$$
    restricts to an isomorphism $\mathbf{S}(F) \to \mathbf{S}(F \otimes L)$.
\end{lem}
\begin{proof}
 Let $T$ be a scheme, and let $(h,f) \in A(T) \times \hat{A}(T)$ be such that 
 $$t_h^*F_T \otimes \mathscr{P}_{f}^{-1} \simeq F_T.$$
 Then
 $$t_h^*(F\otimes L)_T \otimes \mathscr{P}_{\phi_L(h)+f}^{-1}\simeq (F\otimes L)_T,$$
  so that $f_L$ restricts to a morphism $\mathbf{S}(F) \to \mathbf{S}(F\otimes L)$.  Tensoring with $L^{-1}$ induces the inverse map.
 
\end{proof}

\begin{lem} \label{1 l: Poincare S(F)}
Let $A$ be an abelian variety over an algebraically closed field $k$.  Let $F \in \Db(A)$ be a simple object.  Then the isomorphism
$$\begin{pmatrix}
    0 & -1 \\
    1 & 0
\end{pmatrix}:A \times \hat{A} \to \hat{A} \times A$$
restricts to an isomorphism $\mathbf{S}(F) \to \mathbf{S}(\Phi_{\mathscr{P}}(F))$.
\end{lem}
\begin{proof}
   By repeating Mukai's original argument (see \cite[(3.1)]{Mu78}) one sees that for any scheme $T$, any $(h,f) \in A(T) \times \hat{A}(T)$, and any $G \in D(\text{Mod}({\mathcal{O}_{A_T}}))$, 
    $$\Phi_{\mathscr{P}_T}(t_h^*G \otimes \mathscr{P}_f)\simeq t_{f}^*\Phi_{\mathscr{P}_T}(F)\otimes \mathscr{P}^{-1}_{\hat{h}}.$$
    Since the formation of $\Phi_\mathscr{P}$ commutes with base change (see \cite[Proposition 1.3]{Mu87}), it follows that for any $(h,f)$,
    $$t_h^*F_T \otimes \mathscr{P}_{f}^{-1} \simeq F_T \iff t_{-f}^*\Phi_{\mathscr{P}}(F)_T \otimes\mathscr{P}_{\hat{h}}^{-1} \simeq \Phi_{\mathscr{P}}(F)_T$$
    The desired isomorphism follows from sheafifying.
    
\end{proof}

\begin{proof}[Proof of \Cref{1 p:ext}]
    We may reduce to the case where $k$ is algebraically closed and where $F$ is a simple sheaf on $A$.  Our strategy is to reduce to the case where $F$ is locally free, where the result has already been proven (\cite[Proposition 3.16, Theorem 5.8]{Mu78}).  The fibers of the family $\mathscr{P} \otimes p_1^*F$ of sheaves on $A$ have uniformly bounded Castelnuovo-Mumford regularity, so by cohomology and base change we obtain that for any sufficiently ample invertible sheaf $L$ on $A$, $\widehat{F \otimes L}:=\Phi_\mathscr{P}(F\otimes L)$ is locally free.  Moreover, 
    \begin{equation}  
    \text{dim}(\text{Ext}^1(\widehat{F \otimes L},\widehat{F \otimes L}))=\text{dim}(\text{Ext}^1(F\otimes L, F \otimes L))=\text{dim}(\text{Ext}^1(F,F)).
        \end{equation}
    By \Cref{1 l: Poincare S(F)} and \Cref{1 l: tensor S(F)}, we also have   \begin{equation}(\mathbf{S}{(\widehat{F \otimes L} }))\simeq (\mathbf{S}(F \otimes L) ) \simeq  (\mathbf{S}(F)).\end{equation}
    (1) and (2) allow us to apply \cite[Proposition 3.16, Theorem 5.8]{Mu78} to obtain the desired result.
\end{proof}
\begin{cor}{\label{1 c: S(F) is abelian variety}}
Let $F$ be a simple semi-homogeneous sheaf on $X$.  Then $\mathbf{S}(F)$ is a $g$-dimensional abelian variety.
    
\end{cor}
\begin{proof}
    We may reduce to the case where $F$ is a simple semi-homogeneous sheaf on an abelian variety over an algebraically closed field.  In this case, we've shown in the proof of \Cref{1 p:ext} that if $F$ is a simple semi-homogeneous sheaf on $A$, then there exists a simple semi-homogeneous vector bundle $E$ on $\hat{A}$ such that $\mathbf{S}(F) \simeq \mathbf{S}(E)$, and by \cite[Proposition 7.7]{Mu78}, $\mathbf{S}(E)$ is a $g$-dimensional abelian variety.
\end{proof}

\begin{cor}\label{1 c: base change}
    Let $F$ be a simple sheaf on $X$.  Then $F$ is semi-homogeneous if and only if $F_{\overline{k}}$ is.
\end{cor}
\begin{proof}
    Immediate.
\end{proof}
Note that by \Cref{1 c: base change}, our definition of semi-homogeneity is equivalent to that of \cite{dJO22}.
\begin{ex}\label{1 e: basic examples} 
Let $L$ be a line bundle on $X$.  Then $\mathbf{S}(F)$ is semi-homogeneous.  Its stabilizer is the graph of the symmetric homomorphism $\phi_L:A \to \hat{A}$ which sends a point $a \in A(k)$ to the image of $t_a^*L \otimes L^{-1} \in \text{Pic}^0_{X/k}(k)$ under the canonical isomorphism $\text{Pic}^0_{X/k} \simeq \hat{A}$. 
\end{ex}
Using \Cref{1 p:ext}, we obtain that the stack $\Ssh$ of simple semi-homogeneous sheaves is an open substack of $\Spl$.  We are mainly interested in the substacks $\Sshx$ whose objects are those families whose geometric fibers have Chern characters equal to (the pullback of) $\xi \in \text{CH}_{\text{num}}(X_{\overline{k}})_{\mathbb{Q}}$.  When they are nonempty, these stacks are $\mathbb{G}_m$-gerbes over algebraic spaces which we denote by $\sshx$. 

For the remainder of this section, we fix some ${\xi} \in \text{CH}_{\text{num}}(X_{\overline{k}})_{\mathbb{Q}}$ such that $\Sshx$ is nonempty.  It's easy to see that the action of $A \times \hat{A}$ on $\spl$ restricts to an action on $\ssh$.  We will show that we can restrict further to an action on $\sshx$ and that this action makes $\sshx$ a into a torsor under an abelian variety.

\begin{lem}\label{1 l: action and geometric points}
  Let $A$ be an abelian variety over an algebraically closed field $k$.  Let $F$ and $G$ be simple semi-homogeneous sheaves on $A$.  Then $F$ and $G$ have the same numerical Chern character if and only if there exists $(a,\alpha) \in A(k) \times \hat{A}(k)$ such that
  $$G \simeq t^*_{a}F \otimes \mathscr{P}^{-1}_{\alpha}.$$
\end{lem}

\begin{proof}
When $F$ and $G$ are locally free, this follows immediately from \cite[Theorem 7.11(2)]{Mu78}.  We will reduce to this case using the Fourier transform.  First note that we may replace $F$ and $G$ by $F \otimes L$ and $G \otimes L$ for any line bundle $L$.  Choosing $L$ sufficiently ample, we may assume that $\Phi_{\mathscr{P}}(F)$ and $\Phi_{\mathscr{P}}(G)$ are locally free.  It's clear that $\Phi_{\PP}(F)$ and $\Phi_{\PP}(G)$ have the same numerical Chern character if and only if $F$ and $G$ do, and arguing exactly as in \Cref{1 l: Poincare S(F)}, we see that
$t_a^*F \otimes \PP_{\alpha}^{-1} \simeq F$ if and only if $t_{-\alpha}^*\Phi_\PP(F)\otimes \PP_{\hat{a}}^{-1} \simeq \Phi_\PP(F).$
\end{proof}

By \Cref{1 l: action and geometric points}, we see that $A \times \hat{A}$ acts on $\sshx$ by restriction.

\begin{defn}
Let $A \times \hat{A}$ act on $\sshx$ as above.  We define
$$\mathbf{S}(\xi):= \text{Ker}\left((A \times \hat{A}) \to \text{Aut}(\sshx) \right)$$
and
$$A(\xi) := (A \times \hat{A})/\mathbf{S}(\xi).$$    
\end{defn} 

\begin{prop}\label{1 p:torsor}
   $\sshx$ is a torsor under $A(\xi)$.
\end{prop}
 We need the following preliminary lemmas.

\begin{lem}\label{1 l:group actions}
    Let $G$ be a smooth, commutative group scheme of finite type over an algebraically closed field $k$, and let $G$ act on an algebraic space $X$ which is locally of finite type over $k$.  Let $x \in X(k)$, and suppose that the orbit map
    $$\tau_x:G \mapsto X, \quad g \mapsto xg$$
    is surjective.  Then $\tau_x$ induces an isomorphism 
    $$G/\text{stab}(x) \simeq X_{\text{red}}.$$
    \begin{proof}
        Since $G$ is reduced, the scheme-theoretic image of $\tau_{x}$ is $X_{\text{red}}$.  Denote by $\tau_{x,\text{red}}:G \to X_{\text{red}}$ the induced map.  By generic flatness, there exists an open subspace $U \subseteq X_{\text{red}}$ such that the restriction $\tau_{x,\text{red}}^{-1}(U) \to U$ is faithfully flat.  Taking the union over the translates of this map, we obtain a faithfully flat morphism
        $$G=\bigcup_{g \in G(k)} \tau_{x,\text{red}}^{-1}(U)\cdot g \to 
        \bigcup_{g \in G(k)} U\cdot g = X_{\text{red}},$$
        which necessarily factors through the desired isomorphism.
    \end{proof}
\end{lem}
\begin{proof}[Proof of \Cref{1 p:torsor}]
    Since $\Sshx$ is nonempty, we can find an algebraically closed field extension $k'/k$ and $k'$-point $F \in \Sshx(k')$.  It suffices to show that the orbit map
    \begin{equation}
        \tau_{F}:(A \times \hat{A})_{k'} \to \text{SSH}^\xi_{X/k,k'}, \quad (a, \alpha) \mapsto t_{a}^*F \otimes \mathscr{P}_{\alpha}^{-1}
    \end{equation}
    is the quotient by $\mathbf{S}(F)$.  Indeed, this implies that $\mathbf{S}(F)=\mathbf{S}(\xi)_{k'}$ so that (3) yields an $A(\xi)_{k'}$-equivariant isomorphism
    $$A(\xi)_{k'} \simeq \text{SSH}^\xi_{X/k,k'}.$$
    This will show that $\sshx$ is an $A(\xi)$-torsor for the fpqc topology. Since $A(\xi)_{k'}$ will be smooth, $\sshx$ will be too, so it will follow that $\sshx$ is an $A(\xi)$-torsor for the étale topology.
    
    After base change, we may assume that $k=k'$.  \Cref{1 l: action and geometric points} shows that $\tau_{F}$ is surjective.  \Cref{1 l:group actions} implies that $\sshx$ is a $g$-dimensional scheme.  By \Cref{1 p:ext}, we see that $\sshx$ is regular, and therefore reduced, so \Cref{1 l:group actions} gives the desired result.
\end{proof}
\begin{cor}
    $\mathbf{S}(F)=\mathbf{S}({\text{ch}(F)})$ for any simple semi-homogeneous sheaf $F$ on $X$.  In particular, $\mathbf{S}(\xi)$ and $A(\xi)$ are $g$-dimensional abelian varieties
\end{cor}

\begin{cor}
    $\sshx$ is a torsor under an abelian variety.
\end{cor}
We finish with a basic example.
\begin{ex}
    Let $\lambda \in \text{NS}_{X/k}(k)$, and let $L$ be a line bundle on $X_{\overline{k}}$ with $[L]=\lambda_{\overline{k}}\in \text{NS}(X_{\overline{k}})$.  Then $\mathcal{SSH}^{\text{ch}(L)}_{X/k} \simeq \mathscr{P}ic^{\lambda}_{X/k}$.

\end{ex}

\section{Moduli Spaces and Derived Equivalences}\label{2 SECTION}
In this section we produce some twisted derived equivalences between torsors under abelian varieties and their moduli spaces of simple semi-homogeneous sheaves.  We also prove the natural converse to this result and show that many twisted derived equivalences between torsors arise in this way.
\begin{thm} \label{2 t : equivalent to ssh}
    Let $A$ be an abelian variety over a field $k$, let $X$ be an $A$-torsor, and let $\xi \in \text{CH}_{\text{num}}(X_{\overline{k}})_{\mathbb{Q}}$, and suppose that $\mathcal{SSH}^\xi_{X/k}$ is nonempty.  Then there is a twisted derived equivalence
    $$\Phi_\EE:\text{D}^b(\text{SSH}^{\xi}_{X/k},\mu^{-1}) \to \text{D}^b(X),$$
    where $\mu\in \text{Br}(\text{SSH}^{\xi}_{X/k})$ is the universal obstruction and $\EE$ is the universal $p_1^*\mu$-twisted sheaf.
\end{thm}
\begin{proof}
    By \Cref{A l: base change for fields}, we may suppose that $k$ is algebraically closed.  We first verify that conditions (1) and (2) of \Cref{A t: criterion} are satisfied.  We begin by noting that for any closed point $y \in \text{SSH}^\xi_{X/k}$, $\Phi_{\EE}(k(y))=\EE_y$ (see \Cref{A d: twisted skyscraper sheaves} for the definition of $k(x)$).  The first part of (1) is trivially satisfied.  The second part of (1) follows from \cite[Proposition 6.3]{HP24} applied with:
    \begin{enumerate}[(a)]
    \item $S=\Spec(D)$, where  $D=k[\epsilon]/(\epsilon^2)$, $T=\Spec(k)$, and $t$ the usual thickening,
    \item Their $X$ equal to $\mathcal{SSH}^{\xi}_{X/k,D}$ and their $Y$ equal to our $X_{D}$
    \item $F=k(y)$ for a closed point $y\in \text{SSH}^\xi_{X/k}$, and $K=\EE_D$.
    \end{enumerate}
    Since $\EE$ is the universal sheaf, the identity map $T_y\mathcal{SSH}^\xi_{X/k} \to T_y\mathcal{SSH}^\xi_{X/k}$ factors through the map $\text{Lift}[F,t_X] \to \text{Lift}[\Phi_{K_T}(F),t_Y]=T_y\mathcal{SSH}^\xi_{X/k}$.  By a remark in Section $6$ of \cite{HP24}, all vertical maps in the diagram in \cite[Proposition 6.4]{HP24} are isomorphisms, and so is the rightmost diagonal map.  It follows that all maps in this diagram are isomorphisms.  In particular, the leftmost diagonal map, which is the morphism $\Phi_\EE:\Ext^1(k(y),k(y)) \to \Ext^1(\EE_y,\EE_y)$, is an isomorphsim.  Thus condition (1) of \Cref{A t: criterion} is satisfied.  Condition (2) is satisfied by \Cref{2 l: Orlov's lemma}.  This shows that $\Phi_\EE$ is fully faithful.  Using Grothendieck duality (\Cref{A c: coherent twisted adjoints}), we can argue exactly as in \cite[Remarks 3.37 ii)]{Huy06} to show that tensoring with the dualizing complex is a Serre functor.  In this case, the dualizing complexes are trivial line bundles concentrated in degree $-\text{dim}(X)$, so that $\Phi_\EE$ preserves Serre functors.  By \cite[Corollary 1.56]{Huy06}, we obtain that $\Phi_\EE$ is an equivalence.
\end{proof}

\begin{lem}\label{2 l: Orlov's lemma}
    Let $A$ be an abelian variety over an algebraically closed field $k$, let $E$ and $F$ be simple semi-homogeneous sheaves on $A$ such that $\text{ch}(E)=\text{ch}(F)\in \text{CH}_{\text{num}}(A)_{\mathbb{Q}}$.  Then either $E\simeq F$ or $\Ext^i(E,F)=0$ for every $i \in \mathbb{Z}$.
\end{lem}
\begin{proof}
    After tensoring with a sufficiently ample line bundle and applying $\Phi_{\mathscr{P}}$, we may assume that $E$ and $F$ are locally free.  Since $E$ and $F$ have the same Chern character, they have the same slope.  We may apply \cite[Lemma 4.8]{Or02} because its proof does not use the assumptions on the characteristic.
\end{proof}
Using this, we can recover some known examples:

\begin{ex}\label{2 e: Pic equivalences}
     Let $A$ be an abelian variety over a field $k$, and let $X$ be an $A$-torsor.  Let $\lambda \in \text{NS}_{X/k}(k)$, let $\mu=[\mathscr{P}ic^\lambda_{X/k}]\in H^2(\text{Pic}_{X/k}^\lambda,\Gm)$ denote the universal obstruction, and let $\mathscr{P}_{X}$ denote the twisted Poincaré bundle on $\text{Pic}^{\lambda}_{X/k}\times X$.  Then $$\Phi_{\mathscr{P}_{X}}:\Db(\text{Pic}^\lambda_{X/k},\mu^{-1}) \to \Db(X)$$
     is an equivalence.
     \begin{enumerate}
         \item When $X=A$, we obtain \cite[Theorem 4]{AAFH21}
         \item  When $\lambda=0$, we obtain the equivaelence in \cite[Theorem 5]{RR23}.
     \end{enumerate}
\end{ex}

Using \cite{dJO22}, one can obtain the following partial converse to \Cref{2 t : equivalent to ssh}.

\begin{thm}\label{2 t: converse}
    Let $A$ be an abelian variety over a field $k$ and let $X$ be an $A$-torsor.  Let $Y$ be a smooth projective variety, and let $\Phi_{\EE}:D^b(Y,\mu^{-1}) \to \Db(X)$ be an equivalence. If either
\begin{enumerate}
    \item $k$ has characteristic $0$,
    \item $\mu \in Br_1(Y):=\text{Ker}(\text{Br}(Y) \to \text{Br}(Y_{\overline{k}}))$,
    \item $Y$ is a torsor under an abelian variety,
\end{enumerate}
then there exists some $\xi \in \text{CH}_{\text{num}}(X_{\overline{k}})$ such that $Y \simeq \text{SSH}^{\xi}_{X/k}$, $\mu$ is the universal obstruction, and $\EE$ is a shift of the universal sheaf.  In particular, $Y$ is a torsor under an abelian variety.
\end{thm}

First we need two preliminary results.

\begin{lem}\label{2 l: dimension}
    Let $X$ and $Y$ be two smooth, projective varieties over a field $k$.  Let $\mu \in \text{Br}(Y)$, and let $$\Phi_\EE:\Db(Y,\mu^{-1}) \to \Db(X)$$
    be an equivalence.  Then $\text{dim}(X)=\text{dim}(Y)$.
\end{lem}
\begin{proof}
Since $\Phi_{\EE}$ is an equivalence, the right and left adjoints of $\EE$ coincide.  It follows that for any $F \in \Db(X)$,
$$Rp_{1,*}((\EE^\vee \otimes p_1^*\omega_X[\text{dim}(X)]) \otimes^L p_2^*F) \simeq Rp_{1,*}((\EE^\vee \otimes p_2^*\omega_Y[\text{dim}(Y)]) \otimes^L p_2^*F).$$
There exists some $x \in X(k)$ such that the above quanitites are nonzero for $F=k(x)$.  For such $x$, the above isomorphism becomes
$$\Phi_{\EE^\vee}(k(x))\otimes \omega_X[\text{dim}(X)] \simeq \Phi_{\EE^\vee}(k(x))[\text{dim}(Y)].$$
This implies that $\text{dim}(X)=\text{dim}(Y)$.
\end{proof}

\begin{prop}\label{2 p: open immersion}
    Let $X$ and $Y$ be smooth, projective varieties over a field $k$.  Let $\mu \in \text{Br}(Y)$, let $\YY$ be the $\Gm$-gerbe with cohomology class $\mu$, and let $\EE$ be a $p_1^*\mu$-twisted $\YY$-flat family of simple sheaves on $X$.  Suppose that the functor
    $$\Phi_\EE:\Db(Y,\mu^{-1}) \to \Db(X)$$
    is an equivalence.  Then $\EE$ induces an open immersion $i_\EE:\YY \hookrightarrow \mathcal{S}pl_{X/k}$.
\end{prop}
\begin{proof}
We may reduce to the case where $k$ is algebraically closed by \Cref{A l: base change for fields}.  $i_\EE$ induces a map $j_\EE:Y \to \text{Spl}_{X/k}$.  It suffices to show that the following square is cartesian
$$\begin{tikzcd}
{\YY} \arrow[r, "i_\EE"] \arrow[d] & {\mathcal{S}pl_{X/k}} \arrow[d] \\
{Y} \arrow[r, swap, "j_\EE"]           & {\text{Spl}_{X/k}}          
\end{tikzcd}$$
and that $j_\EE$ is an open immersion.  To do this, we may show that $i_\EE$ induces an isomorphism on inertia and that $j_\EE$ is injective on closed points and an isomorphism on the Zariski tangent spaces of closed points.  The map on inertia is the identity map $\Gm \to \Gm$ because it comes from the action of (the pullback of) $I_{\YY}$ on $\EE$ which is given by scalar multiplication since $\EE$ is $p_1^*\mu$-twisted.  Hence, the induced map $\YY \to j_\EE^*\mathcal{S}pl_{X/k}$ is an isomorphism since it is an isomorphism on inertia.  For any pair of distinct closed points $y_1,y_2 \in Y(k)$, $\EE_{y_1}=\Phi_\EE(k(y_1)) \neq \Phi_{\EE}(k(y_2))=\EE_{y_2}$ since $\Phi_\EE$ is fully faithful.  This means that $j_\EE$ is injective on closed points.  To finish, it suffices to show that for any $k$-point $y \in \YY$, $i_\EE$ induces a surjective map from the Zariski tangent space at $y$ to that at $\EE_y$.  This morphism factors through a surjective map by \cite[Proposition 6.3]{HP24}, namely the map $\text{Lift}[\mathcal{O}_y,\YY \to \YY_D] \to \text{Lift}[\EE_y,\mathcal{S}pl_{X/k} \to \mathcal{S}pl_{X/k,D}]=T_{\EE_y}\mathcal{S}pl_{X/k}$.  Note that the vertical maps in \cite[Proposition 6.3]{HP24} are actually surjective by a remark made in Section $6$ of \cite{HP24}.
\end{proof}

\begin{proof}[Proof of \Cref{2 t: converse}]
    We first to show that, up to shift, $\EE$ is a $\YY$-flat family of simple semi-homogeneous sheaves on $X$, where $\YY \to Y$ is the $\Gm$-gerbe with cohomology class $\mu$.  It suffices to do so after base change, so we may assume that $k=\overline{k}$.  Note that $\Phi_\EE$ is still an equivalence by \Cref{A l: base change for fields}.  In this case, $\EE$ induces a morphism $i_{\EE}:\YY \to s\mathscr{D}_{X/k}$ where $\text{s}\mathscr{D}_{X/k}$ is the stack of simple complexes on $X$.  By \cite[Remark 0.3]{Ina02}, $\text{s}\mathscr{D}_{X/k}$ contains $\mathcal{S}pl_{X/k}$ as an open substack, so to show that, up to shift, $\EE$ is a $\YY$-flat family of simple sheaves on $X$, we show that $\EE_{y}$ is a simple sheaf (concentrated in some fixed degree $n$) for each geometric point $y\in Y$.  Let $y \in Y$ be a geometric point.  After base change, we can assume that $y\in Y(k)$.  Then $\EE_y$ is simple since $\Phi_\EE$ is an equivalence, so we can view $\EE_y$ as a $k$-point of $sD_X$. We also have:
    \begin{enumerate}
         \item $\text{Ext}^i(\EE_{y},\EE_{y})<0$ for $i < 0$,
        \item $\text{Ext}^1(\EE_{y},\EE_{y})=\text{dim}(Y)=\text{dim}(X)$.
    \end{enumerate}
    If $k$ is of characteristic zero, then $\EE_y$ automatically a simple semi-homogeneous sheaf up to shift by \cite[Theorem 1.1]{dJO22}.  In the general case, we first need to show that $\text{dim}(\mathbf{S}(\EE_y))\geq \text{dim}(X)$.  Then we can apply \cite[Theorem 1.1, Remark 1.3, (6.1)]{dJO22} to obtain that $\EE_y$ is a simple semi-homogeneous sheaf up to shift.  Recall that by \cite[Theorem A]{CNS22} the bounded derived category of any abelian category has a unique (up to quasi-equivalence) dg-enhancement.  It follows that $\Phi_\EE$ extends to a quasi-equivalence $\Db_{\text{dg}}(Y,\mu^{-1}) \xrightarrow{\sim}\Db_{\text{dg}}(X)$.  In particular, we obtain an isomorphism of autoequivalence group schemes $\text{Autoeq}(\Db_{\text{dg}}(Y,\mu^{-1})) \simeq \text{Autoeq}(\Db_{\text{dg}}(X))$ (see \cite[Corollary 3.24]{TV07}).  Restricting to connected components of the identity and using \cite[Theorem 2.12]{Ros09}, we get an isomorphism $\text{Autoeq}^0(\Db_{\text{dg}}(Y,\mu^{-1})) \simeq A \times \hat{A}$. $\text{Autoeq}^0(\Db_{\text{dg}}(Y,\mu^{-1}))$ contains $\text{Pic}^0_{Y/k}$ as an an abelian subvariety.  Since $\text{Pic}_{Y/k}^0$ acts trivially on $k(y)$, its image under this isomorphism acts trivially on $\EE_y$ so that $\text{dim}(\text{Pic}_{Y/k}^0) \leq \text{dim}(\mathbf{S}(\EE_y))$.  So if $\text{dim}(\text{Pic}_{Y/k}^0) \geq \text{dim}(X)$, then $\EE_y$ is a simple semi-homogeneous sheaf up to shift.  If $Y$ is a torsor under an abelian variety, then $\text{dim}(\mathbf{S}(\EE_y))\geq \text{dim}(Y)=\text{dim}(X)$ as desired.  If $Y$ is arbitrary and $\mu \in \text{Br}_1(Y)$, then $\text{dim}\text({\Pic}_{Y/k}^0)=\text{dim}(\text{Pic}^0_{X/k})$ by \cite[Theorem A.1]{Hon18}, so the result holds in this case as well.

    Thus far, we have shown that for each geometric point $y \in Y$, there exists an integer $n_y$ such that $\EE_y$ is a simple semi-homogeneous sheaf up to shift.  For each $n\in \mathbb{Z}$, we denote the open substack of $s\mathscr{D}_{X/k}$ parameterizing simple sheaves concentrated in degree $n$ by $\mathcal{S}pl_{X/k}[n]$.  Taking the inverse image under $i_{\EE}$, we obtain an open cover of $\YY$.  If $n_{y_1} \neq n_{y_2}$ for some pair of points $y_1,y_2 \in Y$, then we obtain a pair of disjoint open substacks of $\YY$, which is impossible since $\YY$ is irreducible.  Thus we have shown that, up to shift, $\EE$ is a $\YY$-flat family of simple semi-homogeneous sheaves on $X$; note that we have now returned to working over an arbitrary field.  After shifting if necessary, we can apply \Cref{2 p: open immersion} to obtain an open immersion $\YY \hookrightarrow \mathcal{SSH}_{X/k}$.  Since $\YY$ is irreducible, this morphism factors through an open immersion $\YY  \hookrightarrow \mathcal{SSH}^{\xi}_{X/k}$ for some $\xi \in \text{CH}_{\text{num}}(X_{\overline{k}})_{\mathbb{Q}}$.  This map must be surjective, and therefore an isomorphism, since $Y$ and $\text{SSH}^{\xi}_{X/k}$ are proper, irreducible, and of the same dimension.
\end{proof}

\begin{cor} \label{2 c : FM partner of torsor}
Let $A$ be an abelian variety over a field $k$ and let $X$ be an $A$-torsor.  If $Y$ is a smooth projective variety, and if $\mu \in \text{Br}_1(Y)$ is such that $\text{D}^b(Y, \mu^{-1}) \simeq \text{D}^b(X)$, then $Y$ is a torsor under an abelian variety. 
\end{cor}

\begin{cor} \label{2 c: theorem C}
    Let $A$ and $B$ be abelian varieties over field $k$.  Let $X$ be an $A$-torsor and let $Y$ be a $B$-torsor.  Then there is a twisted derived equivalence $\Db(Y, \mu^{-1}) \simeq \Db(X)$ if and only if there exists an isomorphism $Y \simeq \text{SSH}^{\xi}_{X/k}$ for some $\xi \in \text{CH}_{\text{num}}(X)_{\mathbb{Q}}$ and $\mu$ is the universal obstruction.
 \end{cor}

\section{Isometric Isomorphisms and Lagrangians}\label{3 SECTION}
We begin by recalling the definition of an isometric isomorphism:
\begin{defn} Let 
$$f=\begin{pmatrix}
    x & y \\
    z & w
\end{pmatrix}: A \times \hat{A} \rightarrow B \times \hat{B}$$ 
be an isomorphism.  We say that $f$ is \textit{isometric} if
$$f^{-1}={\begin{pmatrix}
    \hat{w} & -\hat{y} \\
    -\hat{z} & \hat{x}
\end{pmatrix}}.$$
\end{defn}
We find it most convenient to deal with isometric isomorphisms using the language of symplectic biextensions (see \cite{Pol96}, \cite[Section 2]{Pol02}, \cite{Pol12}, \cite{MP17}), the essentials of which we recall here.

Following \cite{MP17}, we recall the following definition.  Let $A$ and $B$ be abelian varieties over a field $k$.  A \textit{biextension} of $A \times B$ (by $\Gm$) is a divisorial correspondence from $A$ to $B$, i.e., a line bundle $\LL$ on $A \times B$ together with rigidifications $$\epsilon_A:\LL\mid_{A \times 0} \xrightarrow{\sim} \mathcal{O}_A, \quad \epsilon_{B}:\LL\mid_{0 \times B} \to \mathcal{O}_B$$
whose restrictions to $0 \times 0$ coincide.  Note that any biextension $\LL$ determines a homomorphism $\psi_\LL:A \to \hat{B}$ and that $\LL$ is uniquely determined by $\psi_\LL$.

We consider the case where $A=B=C$.  We say a biextension $\LL$ of $C^2$ is a \textit{skew-symmetric} if $\widehat{\psi_{\LL}}=-\psi_\LL$, and we call a skew-symmetric biextension \textit{symplectic} if $\psi_\LL$ is an isomorphism.  We call an abelian subvariety $i:Z \hookrightarrow C$ \textit{isotropic} (with respect to $\LL$) if the $\hat{i}\psi_{\LL}i=0$.  An isotropic subvariety $Z \subseteq C$ is called \textit{Lagrangian} if $\psi_{\LL}$ restricts to an isomorphism $Z \rightarrow \text{ker}(\hat{i})$.  Note that for any Lagrangian subvariety $Z \subset C$, there is an isomorphism $C/Z \simeq \hat{Z}$.  We call a Lagrangian $Z \subseteq C$ \textit{split} if the projection map
$$p:C \rightarrow C/Z \simeq \hat{Z}$$
admits a section.

We will only be interested in the case where $C=A \times \hat{A}$ and $\LL=\LL_A:= p_{14}^*\PP\otimes p_{32}^*\PP^{-1}$.  In this case, we will write $\psi_{A}$ for $\psi_{\LL_A}$.  It will be useful to note that $A$ and $\hat{A}$ are split Lagrangian abelian subvarieties of $A \times \hat{A}$.  The aim of this section is to produce following description of the Lagrangian subvarieties of $A \times \hat{A}$:
\begin{prop}\label{3 p: lagrangian classification}
    Let $A$ be an abelian variety over a field $k$ and let $X$ be an $A$-torsor.  Then
    \begin{enumerate}
        \item If $\xi \in \text{CH}_{\text{num}}(X_{\overline{k}})_{\mathbb{Q}}$ is such that $\mathcal{SSH}^{\xi}_{X/k}$ is nonempty, then $\mathbf{S}(\xi)\subseteq A \times \hat{A}$ is a Lagrangian abelian subvariety.
        \item If $Z \subseteq A \times \hat{A}$ is a Lagrangian abelian subvariety, then there exists some $\xi \in  \text{CH}_{\text{num}}(X_{\overline{k}})_{\mathbb{Q}}$ such that $\mathcal{SSH}^{\xi}_{X/k}$ is nonempty and $Z=\mathbf{S}(\xi)$.
    \end{enumerate}
\end{prop} 
We begin with some basic facts.

\begin{lem} Let $f:A \times \hat{A} \rightarrow B \times \hat{B}$ be an isomorphism.  Then $f$ is isometric if and only if $(f \times f)^*\LL_{B}\simeq \LL_A$.
\end{lem}
\begin{proof}
  The biextensions $\LL_A$, $\LL_B$, and $(f \times f)^*\LL_B$ are completely determined by their induced homomorphisms:
$$\psi_{A}:A \times \hat{A} \rightarrow \widehat{A \times \hat{A}},\quad (a,a') \mapsto \LL_A\mid_{(a,a')\times A \times \hat{A}}$$
$$\psi_{B}:B \times \hat{B} \rightarrow \widehat{B \times \hat{B}},\quad (b,b')\mapsto \LL_B\mid_{(b,b')\times B \times \hat{B}}$$
$$\phi_{B}:A \times \hat{A} \rightarrow \widehat{A \times \hat{A}},\quad (b,b')\mapsto ((f\times f)^*\LL_B) \mid_{(a,a')\times A \times \hat{A}}.$$
We note that $\phi_B=\hat{f} \circ \psi_B \circ f$, so
$$(f \times f)^*\LL_B \simeq \LL_A \iff \hat{f} \circ \psi_B \circ f =\psi_A \iff \hat{f}\circ \psi_B=\psi_A\circ f^{-1}$$
Letting $f=\begin{pmatrix}
    x & y \\
    z & w
\end{pmatrix}$ and making the identifications, $\widehat{A \times \hat{A}}\simeq A \times \hat{A}$, $\widehat{B \times \hat{B}}\simeq B\times \hat{B}$, we may write
$$\hat{f}=\begin{pmatrix}
    \hat{w} & \hat{y} \\
    \hat{z} & \hat{x}
\end{pmatrix},\quad f^{-1}= \begin{pmatrix}
 \alpha & \beta \\
 \gamma & \delta
\end{pmatrix}$$
and
$$\psi_A=\begin{pmatrix}
    1 & 0 \\
    0 & -1
\end{pmatrix}, \quad \psi_B=\begin{pmatrix}
    1 & 0 \\
    0 & -1
\end{pmatrix}.$$
Multiplying matrices, we obtain the desired result.
\end{proof}

\begin{lem}\label{3 l: image of lagrangians} Let $f:A \times \hat{A}\rightarrow B \times \hat{B}$ be an isometric isomorphism.  Let $Z \subseteq A \times \hat{A}$ be an abelian subvariety.  Then $Z$ is isotropic (resp.~ Lagrangian, resp.~ split Lagrangian) if and only if $f(Z)$ is.
\end{lem}
\begin{proof}
    Let $h:Z \rightarrow f(Z)$ denote the isomorphism obtained by restricting $f$ to $Z$.  Since $f$ is isometric, we have that $\hat{f}\psi_B f =\psi_A$.  Hence, we have a commutative diagram
    $$
\begin{tikzcd}
Z \arrow[r, "i"] \arrow[d, "h"] & A \times \hat{A} \arrow[r, "\psi_A"] \arrow[d, "f"] & B \times \hat{B} \arrow[r, "\hat{i}"] \arrow[d, "\hat{f}^{-1}"] & \hat{Z} \arrow[d, "\hat{h}^{-1}"] \\
f(Z) \arrow[r, "j"]        & B \times \hat{B} \arrow[r, "\psi_B"]                & A \times \hat{A} \arrow[r, "\hat{j}"]                           & \widehat{f(Z)}                   
\end{tikzcd},$$
where $i$ and $j$ denotes the inclusions, and the vertical arrows are isomorphisms.  Combining arrows, we obtain a commutative diagram
$$
\begin{tikzcd}
0 \arrow[r] & Z \arrow[r, "i"] \arrow[d, "h"] & A \times \hat{A} \arrow[r, "\hat{i}\psi_A"] \arrow[d, "f"] & \hat{Z} \arrow[d, "\hat{h}^{-1}"] \arrow[r] & 0 \\
0 \arrow[r] & f(Z) \arrow[r, "j"]        & B \times \hat{B} \arrow[r, "\hat{j}\psi_B"]                & \widehat{f(Z)} \arrow[r]                    & 0
\end{tikzcd},$$
where the vertical arrows are isomorphisms.
We see that $\hat{i}\psi_Ai=0$ if and only if $\hat{j}\psi_B j=0$.  Moreover, the top sequence is short (resp.~ split) exact if and only if the bottom is.
\end{proof}
\begin{ex}\label{3 e: main example} When $k$ is algebraically closed, these examples are given in \cite[Examples 2.2.4]{Pol12}.  The same arguments establish them in general.
\begin{enumerate}
    \item   Let $\phi: A \to \hat{A}$ be a symmetric homomorphism, and let $m \in \mathbb{Z}$.  Then the scheme-theoretic image of the morphism $([m],\phi):A \to A \times \hat{A}$ is a Lagrangian abelian subvariety.
    \item Let $Z \subseteq A \times \hat{A}$ be a Lagrangian abelian subvariety, and suppose that the scheme-theoretic intersection $Z \cap A$ is finite.  Then there exists a symmetric homomorphism $\phi:A \to \hat{A}$ and integer $m \in \mathbb{Z}$ such that $Z=\text{Im}([m],\phi)$.
\end{enumerate}
\end{ex}
Using \Cref{3 e: main example}, we can establish \Cref{3 p: lagrangian classification} for vector bundles.  More precisely:

\begin{lem}\label{3 l: vector bundle lagrangians}
    Let $A$ be an abelian variety over a field $k$ and let $X$ be an $A$-torsor.  Then
    \begin{enumerate}
        \item If $\xi \in \text{CH}_{\text{num}}(X_{\overline{k}})_{\mathbb{Q}}$ is such that $\mathcal{SSH}^{\xi}_{X/k}(\overline{k})$ contains a vector bundle, then $\mathbf{S}(\xi)\subseteq A \times \hat{A}$ is a Lagrangian abelian subvariety whose intersection with $\hat{A}$ is finite.
        \item If $Z \subseteq A \times \hat{A}$ is a Lagrangian abelian subvariety such that $Z \cap \hat{A}$ is finite, then there exists some $\xi \in  \text{CH}_{\text{num}}(X_{\overline{k}})_{\mathbb{Q}}$ such that $\mathcal{SSH}^{\xi}_{X/k}(\overline{k})$ contains a vector bundle and $Z=\mathbf{S}(\xi)$.
    \end{enumerate}
    \end{lem}
    \begin{proof}
        We begin with (1).  It suffices to prove this after base change, so we can assume that $k=\overline{k}$  and $X=A$.  By assumption, there exists a simple semi-homogeneous vector bundle $E$ on $A$ such that $\text{ch}(E)=\xi$.  By \cite[Theorem 7.11 (3)]{Mu78}, we know that there exists a symmetric homomorphism $\phi:A \to \hat{A}$ and an integer $m \in \mathbb{Z}$ such that $\mathbf{S}(\xi)=\mathbf{S}(E)=\text{Im}([m],\phi)$.  It's clear that $S(\xi) \cap \hat{A}$ is finite.

        It suffices to prove (2) after base change, so we may assume that $k=\overline{k}$ and $X=A$.  By \Cref{3 e: main example}, there exists a symmetric homomorphism $\phi:A \to \hat{A}$ and an integer $m \in \mathbb{Z}$ such that $Z=\text{Im}([m],\phi)$.  We may write $\phi=\phi_{L}$ for some line bundle $L$ on $A$.  Let $\delta=\frac{[L]}{m} \in \text{NS}(A)\otimes \mathbb{Q}$.  By \cite[Theorem 7.11]{Mu78}, there exists a simple semi-homogeneous vector bundle $E$ on $A$ such $\mathbf{S}(E)=Z$.  Let $\xi=\text{ch}(E) \in \text{CH}_{\text{num}}(A)_{\mathbb{Q}}$.
    \end{proof}
Using \Cref{3 l: vector bundle lagrangians} together with the following immediate corollary of \cite[Lemma 2.2.7]{Pol12}, we will obtain a proof of \Cref{3 p: lagrangian classification}.

\begin{lem}\label{3 l: finite intersection}
    Let $A$ be an abelian variety over a field $k$, let $L$ be an ample line bundle on $A$, and let $Z \subseteq{A \times \hat{A}}$ be a Lagrangian abelian subvariety.  Then for all but finitely many $n \in \mathbb{Z}$, the intersection of graph of $[n]\circ \phi_L$ with $Z$ is finite.
\end{lem}
\begin{proof}
    We may suppose that $k=\overline{k}$.  Let $\Gamma(n\phi_L)$ denote the graph of $[n]\circ \phi_L$.  $Z \cap \Gamma(n\phi_L)$ is finite if and only if $f_\PP(Z) \cap f_{\PP}( \Gamma(n\phi_L))$ is finite, where 
    $$f_{\PP}=\begin{pmatrix}
        0 & -1 \\
        1 & 0 \end{pmatrix}:A \times \hat{A} \to \hat{A} \times \hat{A}.$$
       $f_{\PP}(\Gamma(n\phi_L))$ is the graph of $[-n] \circ \phi_L$, and by \Cref{3 l: image of lagrangians}, $f_\PP(Z)$ is a Lagrangian abelian subvariety of $A \times \hat{A}$.  The result now follows from \cite[Lemma 2.2.7]{Pol12}
\end{proof}

\begin{proof}[Proof of \Cref{3 p: lagrangian classification}]
    We begin with $(1)$.  It suffices to prove this after base change, so we may assume that $k=\overline{k}$ and $X=A$.  Let $F$ be a simple semi-homogeneous sheaf on $A$ such that $\text{ch}(F)=\xi$.  Then for any sufficiently ample line bundle $L$, $\Phi_{\PP}(F \otimes L)=\widehat{F\otimes L}$ is a simple semi-homogeneous vector bundle on $\hat{A}$.  By \Cref{3 l: vector bundle lagrangians}, $\mathbf{S}(\widehat{F \otimes L}) \subseteq \hat{A} \times A$ is a Lagrangian.  By \Cref{1 l: tensor S(F)} and \Cref{1 l: Poincare S(F)}, there exists an isometric isomorphism $\hat{A} \times A \to A \times \hat{A}$ identifying $\mathbf{S}(\widehat{F\otimes L})$ with $\mathbf{S}(F)$.  It follows from \Cref{3 l: image of lagrangians} that $\mathbf{S}(F)\subseteq A\times \hat{A}$ is a Lagrangian abelian subvariety. 

    It suffices to prove (2) after base change, so we may assume that $k=\overline{k}$ and $X=A$.  By \Cref{3 l: finite intersection}, there exists an ample line bundle $L$ on $A$ such that $Z$ has finite intersection of $A$ with the graph of $\phi_{L}$.  In particular, the image of $Z$ under the isometric isomorphism
    $$g=\begin{pmatrix}
    0 & -1 \\
    1 & 0
    \end{pmatrix} \circ \begin{pmatrix}
        1 & 0 \\
    -\phi_L & 1
    \end{pmatrix}:A \times \hat{A} \to \hat{A} \times A$$
    has finite intersection with $A$.  By \Cref{3 l: vector bundle lagrangians}, there exists a simple semi-homogeneous vector bundle on $\hat{A}$ such that $g(Z)=\mathbf{S}(E)$.  Using \Cref{1 l: Poincare S(F)} and \Cref{1 l: tensor S(F)}, we see that $Z=\mathbf{S}(\hat{E}\otimes L^{-1})$.  Since $\widehat{E} \otimes L^{-1}$ has $g$-dimensional stabilizer, we may apply \cite[Theorem 1.1]{dJO22} to obtain that it is a simple semi-homogeneous sheaf up to shift.  After shifting if necessary, we can suppose that $\hat{E}\otimes L^{-1}$ is a sheaf.  Let $\xi=\text{ch}(\hat{E}\otimes L^{-1})$.
\end{proof}
\begin{cor}\label{3 c: image of isometric iso}
    Let $f:A \times \hat{A} \to B \times \hat{B}$ be an isometric isomorphism, and let $Y$ be a $B$-torsor.  Then there exists some $\zeta \in \text{CH}_{\text{num}}(Y_{\overline{k}})$ such that $\mathcal{SSH}^\zeta_{Y/k}$ is nonempty and such that $f$ restricts to an isomorphism $\hat{A} \xrightarrow{\sim} \mathbf{S}(\zeta)$; in particular, $A \simeq B(\zeta)$ for this $\zeta$.
\end{cor}
\begin{proof}
    Note that $f(\hat{A})$ is Lagrangian by \Cref{3 l: image of lagrangians} and apply \Cref{3 p: lagrangian classification}.
\end{proof}

\begin{rem}\label{3 r: factorization}
    An essential ingredient in the proof of \cite[Theorem 4.13]{Or02} is the claim that any isometric isomorphism
    $$f=\begin{pmatrix}
        x & y \\
        z & w
    \end{pmatrix}:A \times \hat{A} \to B \times \hat{B}$$
    can be written the composite of two isometric isomorphisms
    $$f_1=\begin{pmatrix}
        x_1 & y_1 \\
        z_1 & w_1
    \end{pmatrix}:A \times \hat{A} \to C \times \hat{C}, \quad f_2=\begin{pmatrix}
        x_2 & y_2 \\
        z_2 & w_2
    \end{pmatrix}:C \times \hat{C} \to B \times \hat{B}$$
    such that $y_1$ and $y_2$ are isogenies.  As was noted in \cite{MP17}, this is not easy to see, and a proof has not been provided anywhere in the literature.  We note here that this result follows from applying \cite[Lemma 2.2.7(ii)]{Pol12} to the Lagrangian abelian subvariety $f(\hat{A})$.  Doing so, we obtain the desired factorization of $f$ with
    $$f_1=\begin{pmatrix}
        1 & -\phi_L \\
        0 & 1
    \end{pmatrix} \circ f \quad f_2=\begin{pmatrix}
        1 & \phi_L \\
        1 & 0
    \end{pmatrix},$$
    where $L$ is some sufficiently ample line bundle on $\hat{A}$.
\end{rem}

\section{Isometric Isomorphisms and Derived Equivalences}\label{4 SECTION}
First we show how to produce isometric isomorphisms from derived equivalences between torsors.
\begin{prop}\label{4 p: induced iso}
    Let $A$ and $B$ be abelian varieties over a field $k$.  Let $X$ be an $A$-torsor, let $Y$ be a $B$-torsor, and let $\Phi_{\EE}:\Db(X) \to \Db(Y)$ be an equivalence.  Then, up to shift, $\EE$ is a simple semi-homogeneous sheaf on $X \times Y$, and $\mathbf{S}(\EE)$ is the graph of an isomorphism $g:A \times \hat{A} \to B \times \hat{B}$ such that $g \circ(\text{id},[-1]):A \times \hat{A} \to B \times \hat{B}$ is isometric.  We let $f_{\EE}=g \circ(\text{id},[-1])$.
\end{prop}
\begin{proof}
    First we note that $\EE$ is concentrated in a single degree since this is true when $k$ is algebraically closed by \cite[Proposition 3.2]{Or02}.  Shifting if necessary, we assume that $\EE$ is a concentrated in degree 0.  By \Cref{2 t: converse}, $\EE$ is an $X$-flat family of simple semi-homogeneous sheaves on $Y$.  This implies that $\EE$ is a simple sheaf on $X \times Y$.  We claim that the subgroup $\mathbf{S}(\EE)\subset A \times B \times \hat{A} \times \hat{B}$ is the graph of an  isomorphism $g:A \times \hat{A} \to B \times \hat{B}$ such that $g \circ (\text{id}, [-1])$ is isometric.  It suffices to show this after base change, so we may assume that $k=\overline{k}$, $X=A$, and $Y=B$.  By \cite[Corollary 2.13]{Or02}, there exists an isometric isomorphism $f:A \times \hat{A} \to B \times \hat{B}$ such that for any $(a,\alpha) \in A(k) \times \hat{A}(k)$, $f(a,\alpha)=(b,\beta)$ if and only if $(a,b,-\alpha,\beta) \in \mathbf{S}(\EE)(k)$.  It follows that $\text{dim}(\mathbf{S}(\EE))=\text{dim}(A \times B)$, so $\mathbf{S}(\EE)$ is an abelian subvariety of $A \times B \times \hat{A} \times \hat{B}$ equal to $(\text{id},\text{id},[-1],\text{id})\left(\Gamma_{f} \right)$.  We must have $g=f \circ (\text{id},[-1])$. 
    \end{proof}

\begin{rem}
Using methods quite different from ours, an isometric isomorphism is associated to a derived equivalence between two $A$-torsors in \cite[Theorem 5.1]{AKW17}.  It's not clear how our isomorphism relates to theirs.
\end{rem}
\begin{ex}\label{4 e: tensor}
    Let $A$ be an abelian variety over a field $k$, let $L$ be a line bundle on an $A$-torsor $X$.  Then the isometric isomorphism associated to the equivalence $L\otimes(-)=\Phi_{\Delta_{*}L}:\Db(X) \to \Db(X)$ is given by 
    $$\begin{pmatrix}
        1 & 0 \\
        \phi_L & 1
    \end{pmatrix}:A \times \hat {A} \to A \times \hat{A}.$$
\end{ex}

\begin{ex}\label{4 e: counterexample}
    Let $A$ be an abelian variety over a field $k$.  We have a short exact sequence of group schemes
    $$0 \to \hat{A} \to \text{Pic}_{A/k} \to \text{NS}_{A/k} \to 0$$
    and therefore a long exact sequence on étale cohomology
    $$0 \to \hat{A}(k) \to \text{Pic}_{A/k}(k) \to \text{NS}_{A/k}(k) \xrightarrow{c_{-}} H^1(k, \hat{A}) \to \dots$$
    It may be the case that for some $k$-point $\lambda \in \text{NS}_{A/k}(k)$, the element $c_\lambda \in H^1(k, \hat{A})$ is nonzero.  The existence of such a $\lambda$ is equivalent to the existence of a symmetric homomorphism $\phi:A \to \hat{A}$ for which cannot find a line bundle $L$ on $A$ such that $\phi=\phi_L$.  Suppose that such a homomorphism exists and consider the isometric isomorphism
    $$f=\begin{pmatrix}
        1 & 0 \\
        \phi & 1
    \end{pmatrix}:A \times \hat{A} \to A \times \hat{A}.$$
    We claim that $f$ cannot be written as $f_{\EE}$ for some equivalence $\Phi_{\EE}:\Db(A) \to \Db(A)$.  Suppose there were such an $\EE$.  Then by \Cref{4 e: tensor} and \cite[Corollary 3.4]{Or02}, we can find some $(a,\alpha) \in A(\overline{k}) \times \hat{A}(\overline{k})$ and an integer $n\in \mathbb{Z}$ such that
    $$\EE_{\overline{k}}\simeq t_{(a,0)}^*\Delta_*L \otimes \PP_\alpha [n],$$
    where $L$ is a line bundle bundle on $A_{\overline{k}}$ such that $\phi_L=\phi_{\overline{k}}$.  We must have
    $$L\otimes \mathscr{P}_{\alpha}[n]=\Phi_{\EE_{\overline{k}}}(\mathcal{O}_{A_{\overline{k}}})=\Phi_{\EE}(\mathcal{O}_A)_{\overline{k}},$$
    which implies that we can write $\phi=\phi_M$ for some line bundle $M$ on $A$, namely $M=\Phi_{\EE}(\mathcal{O}_A)[-n]$.
\end{ex}
\begin{rem}\label{4 r: counterexample remark}
With \Cref{4 e: counterexample} in mind, one might ask whether or not there exists a pair of $A$-torsors $X$ and $Y$ such and an equivalence $\Phi_{\EE}:\Db(X) \to \Db(Y)$ such that $f_{\EE}=f$, where $f$ is the morphism in \Cref{4 e: counterexample}.
The existence of such $\EE$ implies that $X=Y$.  To show this, it suffices to show that some $\overline{k}$-point $\overline{x} \in X(\overline{k})$, $\EE_{\overline{x}} \simeq k(\overline{y})[n]$ for some point $\overline{y}\in Y({\overline{k}})$ and some integer $n \in \mathbb{Z}$: \Cref{2 t: converse} implies that $X$ is isomorphic to $\text{SSH}^{\text{ch}(k(\overline{y}))}_{Y/k}$, which is isomorphic to $Y$.  We can suppose that $k$ is algebraically closed and that $X=Y=A$.  Then we can argue as in \Cref{4 e: counterexample} to show that
$$\EE \simeq t_{(a,0)}^*\Delta_*L \otimes \PP_\alpha [n]$$
for some $(a,\alpha)\in A(k)\times \hat{A}(k)$, some line bundle $L$ on $A$, and some integer $n \in \mathbb{Z}$.  It's easy to see that $\EE_a$ takes the desired form.  Now that we know that $X=Y$, we can just argue exactly as in \Cref{4 e: counterexample} to show that the existence of $\EE$ is equivalent to the existence of a line bundle $L$ on $X$ such that $\phi=\phi_L$.  Thus, we've reduced the problem to answering a question asked in \cite[Section 4]{PS99}.  Unfortunately this question still has yet to be answered.
\end{rem}

Finally we prove \Cref{0 c: Lagrangians} and \Cref{0 c: isometric iso}.

\begin{prop}\label{3 p: Lagrangian equivalences}
    Let $A$ and $B$ be abelian varieties over a field $k$, and let $X$ be an $A$-torsor.  Then there exists a $B$-torsor $Y$ and an equivalence $\Db(Y,\mu^{-1}) \simeq \Db(X)$ for some $\mu \in \text{Br}(Y)$ if and only if $B$ is isomorphic to a Lagrangian abelian subvariety of $A \times \hat{A}$.
\end{prop}
\begin{proof}
    This follows immediately from \Cref{3 p: lagrangian classification} and \Cref{2 t : equivalent to ssh}.
\end{proof}

\begin{prop} \label{3 p: theorem E}
    Let $A$ and $B$ be abelian varieties over a field $k$, and let $f:A \times \hat{A} \to B \times \hat{B}$ be an isometric isomorphism.  Then for any $A$-torsor $X$, there exists a $B$-torsor $Y$ and a twisted derived equivalence $\Db(Y,\mu^{-1}) \simeq \Db(X)$, where $\mu \in \text{Br}_1(Y)$.  In particular, if $k$ is a finite field, then $\Db(B) \simeq \Db(A)$.
\end{prop}
\begin{proof}
    The first part follows immediately from \Cref{3 c: image of isometric iso} and \Cref{2 t : equivalent to ssh}.  To prove the second statement, suppose that $k$ is finite.  Using \Cref{3 c: image of isometric iso}, we obtain that $\text{SSH}^{\zeta}_{B/k}$ is an $A$-torsor for some $\zeta \in \text{CH}_{\text{num}}(B_{\overline{k}})_{\mathbb{Q}}$.  By Lang's theorem, the torsor $\text{SSH}^{\zeta}_{B/k}$ is trival, so it has a $k$-point.  Since $k$ is finite, $\text{Br}(k)=0$, so this point lifts to a $k$-point of the stack $\mathcal{SSH}^{\zeta}_{B/k}$. Thus we can choose a simple, semi-homogeneous sheaf $E$ on $B$ with $\text{ch}(E_{\overline{k}})=\zeta$.  Consider the orbit map
    $$\tau_{E}:B \times \hat{B} \to \text{SSH}^{\zeta}_{B/k}, \quad (b, \beta) \mapsto t_b^*E \otimes \mathscr{P}_{\beta}^{-1}.$$
    Since $\tau_E$ is induced by the sheaf $(m \circ p_{12})^*E \otimes p_{23}^*\mathscr{P}^{-1}$ on $B \times B \times \hat{B}$, we can find a morphism $\widetilde{\tau_{E}}:B \times \hat{B} \to \mathcal{SSH}^\zeta_{B/k}$ making the following diagram commute
    $$
\begin{tikzcd}
                        & {\mathcal{SSH}^{\zeta}_{B/k}} \arrow[d, "\pi"] \\
{B \times \hat{B}} \arrow[ru, "\widetilde{\tau_E}"] \arrow[r, "\tau_E"] & {\text{SSH}_{B/k}^{\zeta}}          
\end{tikzcd}$$
where the $\pi$ is the projection.  Since $\tau_{E}$ is the quotient by $\mathbf{S}(\zeta)=f(\hat{A})$, \Cref{3 l: image of lagrangians} implies that $\tau_E$ admits a section $\sigma$.  It follows that $\widetilde{\tau_E}\circ \sigma$ is a section of $\pi$, so the universal obstruction $\mu \in \text{Br}(\text{SSH}^{\zeta}_{B/k})$ vanishes.  Thus we can apply \Cref{2 t : equivalent to ssh} and simply note that $A=\text{SSH}^\zeta_{B/k}$.
\end{proof}

\appendix
\section{Twisted Derived Categories}\label{A SECTION}
We recall some basic results on twisted derived categories using the language of gerbes.  Many of the results of the section appear in Căldăraru's thesis \cite{C00}; our main goal is to prove a version of \cite[Theorem 3.2.1]{C00}.  Since our aim is to work over arbitrary fields, we instead mainly refer to \cite{BS21}, \cite{Lie07}, and \cite{HP24}.  We make no claims to originality in this appendix: almost every result here is either proven in one of the aforementioned papers, or it follows immediately from a repetition of their arguments.

\subsection*{Notation and Conventions}
Use the definition of algebraic stacks given in \cite{Stacks}, i.e., we have do not make any separatedness assumptions.  We denote by $\text{Mod}(\mathcal{O}_X)$ (resp.~ \text{Q}Coh(X)) the category of $\mathcal{O}_X$-modules (resp.~ quasi-coherent $\mathcal{O}_X$-modules) on the lisse-étale topos of $X$.  Let $\D(X)$ denote the unbounded derived category of the abelian category of $\text{Mod}(\mathcal{O}_X)$, and let $\Dq(X)$ denote the full subcategory of $\D(X)$ of objects with quasi-coherent cohomology.  We then define $\Dq^*(X)$ as usual, where $* \in \{+,-,b\}$.  
  When $P \in \Dq(X)$, the derived tensor product $(-)\otimes^LP:\D(X) \to \D(X)$ restricts to a functor $(-)\otimes^L P:\Dq(X) \to \Dq(X)$.  We let $R\mathcal{H}om(P,-):\Dq(X) \to \Dq(X)$ denote its right adjoint.  Given a morphism $f:X \to Y$ of algebraic stacks, we let $Lf^*:\Dq(Y) \to \Dq(X)$ denote the functor $Lf_{\text{qc}}^*$ defined in [Section 1.3]\cite{HR17}.  We let $Rf_*:\Dq(X) \to \Dq(Y)$ denote the right adjoint to $Lf^*$.    When $f$ is concentrated, $Rf_*:\Dq(X) \to \Dq(Y)$ has a right adjoint which we denote by $f^\times:\Dq(Y) \to \Dq(X)$ (see \cite[Theorem 4.14]{HR17}).

  All gerbes are for the big fppf topology.   By a $\textit{diagonalizable group}$, we always mean a diagonalizable group scheme which is of finite type over $\Spec(\bbz)$.  Let $\Delta$ be a digonalizable group, and let $X$ an algebraic stack; we refer to a gerbe $\pi:\XX \to X$ which is banded by $\Delta_X:=\Delta \times_{\Spec(\bbz)}X$ a $\Delta$\textit{-gerbe}.

\subsection{Twisted Derived Categories} We begin with some general results on twisted derived categories which are needed to study Fourier-Mukai transforms between them.
\begin{lem}\label{A l: concentrated}
    Let $\Delta$ be a diagonalizable group, let $X$ be an algebraic stack, and let $\pi:\mathscr{X}\to X$ be a $\Delta$-gerbe.  Then $\pi$ is a concentrated morphism. 
\end{lem}
\begin{proof}
    This follows immediately from \cite[Theorem C]{HR15} and the fact that diagonalizable groups are linearly reductive.
\end{proof}
We will be mainly interested in the case where $\Delta=\Gm$.  In this case $\Delta$-gerbes are not proper; nevertheless, we have the following.
\begin{lem}\label{A l: pushforward along a gerbe preserves coherence}
    Let $\Delta$ be a diagonalizable group, and let $\pi:\XX \to X$ be a $\Delta$-gerbe over a noetherian algebraic stack.  Then the restriction of $R\pi_*$ to $D^b_{\text{coh}}(\XX)$ factors through $D^b_{\text{coh}}(X) $.
\end{lem}
\begin{proof}
It suffices to verify this in the case where $\XX\simeq X\times B\Delta$.  In this case, $\pi_*$ sends a sheaf to its subsheaf of $\Delta$-invariants.  This is an exact functor which preserves coherence, so we are done. 
\end{proof}
Let $\Delta$ be a diagonalizable group, let $X$ be an algebraic stack, and let $\pi:\XX \to X$ be a $\Delta$-gerbe.  We assume that the reader is familiar with the notion of a $\chi$-twisted (or $\chi$-homogeneous, using the language of \cite{BS21}) $\mathcal{O}_\XX$-module, where $\chi:\Delta_X \to \mathbb{G}_{m,X}$ is a character (see \cite{Lie07} or \cite{BS21} for the definition).

\begin{prop}\cite[Theorem 4.7]{BS21} \label{A p: qcoh}
    Let $X$ be an algebraic stack and let $\pi:\XX \to X$ be a $\Delta$-gerbe.  Then the full subcategory $\text{QCoh}_\chi(\XX) \subset \text{QCoh}(\XX)$ consisting of $\chi$-twisted sheaves is a Serre subcategory of $\text{QCoh}(\XX)$, where $\chi$ is any character of $\Delta_X$.  Taking the coproduct gives an equivalence
    $$\prod_{\chi \in A}\text{QCoh}_{\chi}(\XX) \xrightarrow{\sim} \text{QCoh}(\XX),\quad (\FF_\chi)_\chi \mapsto \bigoplus_{\chi \in A}\FF_\chi$$
    of abelian categories, where $A$ denotes the character group of $\Delta$.
\end{prop}

\begin{cor} 
Let $X$ be an algebraic stack, let $\pi:\XX \to X$ be a $\Delta$-gerbe.  For each character $\chi$ of $\Delta_X$, let $\text{Coh}_\chi(\XX) \subset \text{QCoh}_{\chi}(\XX)$ denote the full subcategory consisisting of those $\chi$-twisted sheaves which are coherent.  Then $\text{Coh}_\chi(\XX)$ is an abelian category.
\end{cor}

\begin{defn}
    Given a character $\chi:\Delta_X \to \mathbb{G}_{m,X}$, we say that an object $\FF \in \Dq(\XX)$ is $\chi$-twisted if it has $\chi$-twisted cohomology.  We denote by $\Dqch(\XX)$ the full subcategory of $\Dq(\XX)$ consisting of with $\chi$-twisted cohomology.  We define $\Dcch(\XX)$, $\text{D}^*_{\text{qc},\chi}(\XX)$, and $\text{D}^*_{\text{coh},\chi}(\XX)$ similarly, where $* \in \{b,+,-\}$.  We denote by $\Perf_\chi(\XX)$ the full subcategory of $\Perf(\XX)$ consisting of objects with $\chi$-twisted cohomology.
\end{defn}

\begin{prop}\cite[Theorem 5.4]{BS21}\label{A p: splitting}
    Let $X$ be an algebraic stack, and let $\pi: \XX \to X$ be a $\Delta$-gerbe.  Then taking the coproduct defines an equivalence 
    $$\prod_{\chi \in A}\Dqch(\XX) \to \Dq(\XX),\quad (\FF_\chi)_{\chi \in A} \mapsto \bigoplus_{\chi \in A}\FF_\chi,$$
    of triangulated categories, where $A$ denotes the character group of $\Delta$.
\end{prop}

Eventually we will be interested in the case where $X$ is a smooth, projective variety over a field $k$ and $\Delta =\mathbb{G}_m$.  As in the untwisted case, one typically works with the bounded derived category of the abelian category $\text{Coh}_{\text{id}}(\XX)$ as opposed to $\text{D}^b_{\text{coh}, \text{id}}(\XX)$.  We show that, as in the untwisted case, the two categories are equivalent, and that they are both equivalent to $\Perf_{\text{id}}(\XX)$.

\begin{prop} \label{A p: D(qcoh)}
    Let $\Delta$ be a diagonalizable group, let $X$ be quasi-compact, quasi-separated algebraic space, and let $\pi:\XX \to X$ be a $\Delta$-gerbe.  Then the canonical functor
    $$\text{D}(\text{QCoh}(\XX)) \to \Dq(\XX)$$
    is an equivalence of triangulated categories.
\end{prop}
\begin{proof}
    Apply \cite[Proposition 6.14(1)]{AHR23} and \cite[Theorem 1.2]{HNR19}.
\end{proof}

\begin{rem}
    When $\Delta=\Gm$ this is proven for $X$ a quasi-compact algebraic stack with quasi-finite and separated diagonal in \cite[Example 9.3]{HR17}.
\end{rem}

\begin{cor} \label{A c: chi twisted D(qcoh)}
Let $\Delta$ and $\pi$ be as in \Cref{A p: D(qcoh)}.  Let $\chi$ be a character for $\Delta$.  Then the canonical functors
$$\text{D}(\text{QCoh}_{\chi}(\XX)) \to \text{D}_\chi(\text{QCoh}(\XX)) \to \Dqch(\XX) $$
are equivalences of triangulated categories,
where $\text{D}_\chi(\text{QCoh}(\XX))$ denotes the full subcategory of $\text{D}(\text{QCoh}(\XX))$ consisting of objects with $\chi$-twisted cohomology.
\end{cor}
\begin{proof}
    The last equivalence is an immediate consequence of \Cref{A p: D(qcoh)}.  Let $\GG^\bullet$ be a complex of quasi-coherent sheaves with $\chi$-twisted cohomology.  For each $\GG^n$, let $\GG^n_\chi$ be as in \Cref{A p: qcoh}.  There is a natural map $\GG^\bullet_\chi \to \GG^\bullet$ which is clearly a quasi-isomorphism.
\end{proof}
\begin{cor} \label{A c: D(coh)}
Let $\Delta$, $\pi$, and $\chi$ be as in \Cref{A c: chi twisted D(qcoh)}, and suppose that $X$ is noetherian.  Then the canonical functors
$$\text{D}^b(\text{Coh}_{\chi}(\XX)) \to \Dbc(\text{QCoh}_{\chi}(\XX)) \to \Dbcch(\text{QCoh}(\XX) \to \Dbcch(\XX)$$
are equivalences of triangulated categories,
where $\Dbc(\text{QCoh}_{\chi}(\XX))$ denotes the full subcategory of $\text{D}^b(\text{QCoh}_{\chi}(\XX))$ consisting of objects with coherent cohomology.
\end{cor}
\begin{proof}
    By \Cref{A c: chi twisted D(qcoh)}, we only need to establish the first equivalence.  We argue exactly as in \cite[Propoosition 3.5]{Huy06}.  Let $\GG^\bullet$ be a bounded complex of $\chi$-twisted quasi-coherent $\mathcal{O}_{\XX}$-modules
    $$\dots \rightarrow 0 \rightarrow \GG^n \xrightarrow{d^n} \dots \xrightarrow{d^{m-1}} \GG^m \rightarrow 0 \rightarrow \dots$$
    with coherent $\chi$-twisted cohomology $\mathcal{H}^i$.  Apply \Cref{A l : surjections} below to the surjections
    $$d^j:\GG^j \twoheadrightarrow \text{Im}(d^j) \subset \GG^{j+1}, \quad \text{ker}(d^j) \twoheadrightarrow \mathcal{H}^j$$
    to obtain coherent subsheaves $\GG^j_1$ and $\GG^j_2$, respectively.  Now replace $\GG^j$ be the  coherent sheaf generated by $\GG^j_i$, $i=1,2$, and $\GG^{j-1}$ by the pre-image of the new $\GG^j$ under $\GG^{j-1} \to \GG^j$.  The inclusion defines a quasi-isomorphism, and now $\GG^i$ is coherent for $i \geq j$.  Note that by \Cref{A p: qcoh}, any quasi-coherent subsheaf of a $\chi$-twisted quasi-coherent sheaf is $\chi$-twisted.
\end{proof}

\begin{lem}\label{A l : surjections}
    Let $\XX$ be a noetherian algebraic stack, let $\FF$ be a quasi-coherent $\mathcal{O}_\XX$-module, let $\GG$ be a coherent $\mathcal{O}_\XX$-module, and let $\theta:\FF \twoheadrightarrow \GG$ be a surjective morphism of $\mathcal{O}_\XX$-modules.  Then there exists a coherent $\mathcal{O}_\XX$-submodule $\FF' \subseteq \FF$ such that the restriction of $\theta$ to $\FF'$ is surjective.
\end{lem}
\begin{proof}
We can find a smooth surjective map $\phi:U=\Spec(A) \to \XX$ such that $\phi^*\GG$ is generated by global sections $s_1,\dots,s_n$.  It suffices to show that there exists a coherent subsheaf $\FF'\subseteq \FF$ such that $\phi^*\FF' \to \phi^*\GG$ is surjective. By \cite[Proposition 15.4]{LMB00}, $\FF$ is the filtered colimit of its coherent subsheaves.  It follows that for each $1 \leq i  \leq n$, we can choose a finite covering $\psi_{j}:V_{ij} \to U$ and coherent subsheaves $\FF_{ij} \subseteq \FF$ such that $\psi_{j}^*s_{i}$ is in the image of $\FF_{ij}(V_{ij}) \to \GG(V_{ij})$.  We can take $\FF'$ to be the subsheaf of $\FF$ generated by all of the $\FF_{ij}$. 
\end{proof}

\begin{lem}\label{A l: same} Let $X$ be a regular noetherian quasi-separated algebraic space, let $\Delta$ be a diagonalizable group, let $\pi:\XX \to X$ be a $\Delta$-gerbe, and let $\chi$ be a character of $\Delta$.  Then the canonical functor
$$\Perf_\chi(\XX) \to \Dbcch(\XX) \simeq D^b(\text{Coh}_{\chi}(\XX))$$
is an equivalence of triangulated categories.
\end{lem}
\begin{proof}
  It suffices to show that that any object of $\Dbcch(\XX)$ is perfect.  Since $\pi:\XX \to X$ is smooth, it admits sections étale locally on $X$, and therefore the gerbe $\XX \to X$ is locally trivial in the lisse-étale topology.  So we may reduce to the case where $\XX \simeq X \times B\Delta$.  Let $\mathscr{L}$ denote the pullback of the standard representation of $\Gm$ along $B\chi:X \times B\Delta \to X \times B\Gm$.  By \Cref{A l:tensor}, tensoring with $\mathscr{L}$ gives an equivalence of categories $\text{D}^b_{\text{coh}, 0} (\XX) \xrightarrow{\sim}\Dbcch(\XX)$, where we denote the trivial character $\Delta$ by $0$.  It suffices to verify that a complex is perfect after tensoring with line bundle, so we may reduce to the case where $\chi=0$.  Pulling back along the section $X \to \XX$ gives an equivalence of categories $\Dbc(X) \xrightarrow{\sim} \text{D}^b_{\text{coh},0}(\XX)$, so it suffices to establish the result for $X$.  In fact, it suffices to establish the result for any quasi-compact scheme $\tilde{X}$ with a surjective étale map $\tilde{X} \to X$.  In this case the result is well-known (e.g., see \cite[Proposition 3.26]{Huy06}).
\end{proof}

\begin{rem}\label{A r: notation}
In the case where $\Delta=\Gm$, $\XX$ is a $\Delta$-gerbe, and $\chi=\text{id}$, we will typically refer to $\chi$-twisted sheaves as $[\XX]$-twisted sheaves and write $\Db(X,[\XX])$ instead of $\Db(\text{Coh}_{\text{id}}(\XX))$, where $[\XX] \in H^2(X,\Gm)$ denotes the cohomology class corresponding to $\XX$.
\end{rem}
 Next we recall how some of the standard derived functors interact with the decomposition given in \Cref{A p: decomp}.  Here we only explain what is needed to work with twisted Fourier-Mukai transforms (e.g., we do not define the tensor product between two twisted sheaves living on different gerbes); we refer the interested reader to \cite{BS21} for more details.

\begin{lem} \cite[Theorem 5.4]{BS21} \label{A l:tensor}
    Let $\Delta$ be a diagonalizable group, and let $\pi:\XX \to X$ be a $\Delta$-gerbe.   Suppose that $\FF_\chi$ and $\GG_\psi$ are objects in $\Dq(\XX)$ which are homogeneous for the characters $\chi$ and $\psi$.  Then the objects
    $$\FF_\chi\otimes^L\GG_\psi,\quad R\mathcal{H}om(\GG_\psi,\FF_\chi)$$
    are $\chi\psi$- and $\chi\psi^{-1}$-homogeneous, respectively.
\end{lem}

\begin{prop}\label{A p: decomp}
    Let $X$ be an algebraic stack, and let $\pi:\XX \to X$ be a gerbe banded by a diagonalizable group $\Delta$.  Let $f:Y \to X$ be a morphism, and let $g:\YY \to \XX$ denote the pullback of $f$ along $\pi$.  Then the functors
    $$Lg^*:\Dq(\XX)\to \Dq(\YY),\quad Rg_*:\Dq(\YY)\to\Dq(\XX)$$
    respect the decompositions of $\Dq(\XX)$ and $\Dq(\YY)$ given in \Cref{A p: splitting}.  If $f$ is concentrated, then $g^\times$ exists and it respects these decompositions.
\end{prop}
\begin{proof}
    The statement about $Lg^*$ and $Rg_*$ is \cite[Propoisiton 5.6]{BS21}.  The existence of $g^\times$ is \cite[Theorem 4.14(1)]{HR17}.  $g^\times$ respects the decomposition of \Cref{A p: splitting} because $g_*$ does and by uniqueness of right adjoints.
\end{proof}
We are interested in the following special case:
\begin{cor}\label{A c: push}
    Let $S$ be an algebraic stack, and let $X$ and $Y$ be algebraic stacks over $S$.  Let $\Delta$ be a diagonalizable group, let $\pi_\XX:\XX \to X$ and $\pi_\YY:\YY \to Y$ be $\Delta$-gerbes, and let $\chi$ and $\psi$ be characters.  Then $\XX \times_S \YY \to X\times_S Y$ is a $\Delta\times \Delta$-gerbe, and the restriction of $Rp_{\XX,*}:\Dq(\XX\times_S \YY) \to \Dq(\XX)$ to $\text{D}_{\text{qc},(\chi,\psi)}(\XX)$ factors through $\Dqch(\XX)$.  In fact, if $\psi \neq 0$, and $\FF \in \text{D}_{\text{qc},(\chi,\psi)}$, then $Rp_{\XX,*}\FF \simeq 0$.
\end{cor}
\begin{proof}
    We have a cartesian square
$$\begin{tikzcd}
{\XX\times_S \YY} \arrow[r, "\pi_\XX \times \text{id}"] \arrow[d, swap, "\text{id}\times \pi_\YY"] & {X \times_S \YY} \arrow[d, "\text{id}\times \pi_\YY"] \\
{\XX \times_S Y} \arrow[r, "\pi_\XX \times \text{id}"] \arrow[d, swap, "\text{pr}_\XX"] & {X \times_S Y} \arrow[d, "\text{pr}_{X}"] \\
{\XX} \arrow[r, "\pi_\XX"]           & {X}          
\end{tikzcd}$$
so by the \Cref{A p: decomp}, it suffices to show that any $(\chi,\psi)$-homogeneous object is $\chi$-homogeneous for the action of $I_{\XX\times \YY/X \times \YY}$.  It follows from the definitions that it suffices to consider the case of a sheaf.  Let $\FF$ be a $(\chi,\psi)$-homogeneous sheaf.  Then we have a commutative diagram
$$\begin{tikzcd}
{I_{\XX\times \YY / X \times Y} \times \FF} \arrow[r] \arrow[d, swap, "(\chi\text{,}\psi) \times \text{id}"] & {\FF} \arrow[d, "\text{id}"] \\
{\Gm \times \FF} \arrow[r]           & {\FF}          
\end{tikzcd}$$
where the horizontal arrows are the usual actions.  Via the banding $I_{\XX\times \YY / X \times Y}\simeq \Gm \times \Gm$, $I_{\XX \times \YY/X \times \YY}$ is identified with the first copy of $\Gm$.  Hence, restricting the above diagram to $I_{\XX \times \YY/X \times \YY}$ we get a commuting square
$$\begin{tikzcd}
{I_{\XX\times \YY / X \times \YY} \times \FF} \arrow[r] \arrow[d, swap, "\chi \times \text{id}"] & {\FF} \arrow[d, "\text{id}"] \\
{\Gm \times \FF} \arrow[r]           & {\FF}          
\end{tikzcd}$$
which tells us that $\FF$ is $\chi$-homogeneous, as desired.  For the second part, note that it suffices to show that $R(\text{id} \times \pi_\YY)_*\FF \simeq 0$.  This follows from \Cref{A l : push}.
\end{proof}

\begin{lem} \label{A l : push}
    Let $\pi:\XX \to X$ be a $\Delta$-gerbe.  Let $P \in \Dqch(\XX)$.  If $\chi \neq 0$, then $R\pi_*P \simeq 0$.
\end{lem}
\begin{proof}
    Since $\pi$ is concentrated, we may apply \cite[Theorem 2.6(2)]{HR17} to compute $R\pi_*\FF$ via the hypercohomology spectral sequence
    $$R^p\pi_*\mathcal{H}^q(\FF)\Rightarrow R^{p+q}\pi_*\FF.$$
    In particular, we may reduce to the case where $\FF$ is a sheaf.  By flat base change \cite[Theorem 2.6(4)]{HR17}, we may reduce to the case where $\XX \simeq X \times B\Delta$.  In this case, $R\pi_*$ coincides with the fixed points functor, and the result is immediate.
\end{proof}

\begin{prop}\cite[Proposition 5.7]{BS21}\label{A p: bands}
    Let $\XX$ and $\XX'$ be gerbes over $X$ banded by diagonalizable groups $\Delta_X$ and $\Delta'_X$ respectively, and let $\rho:\XX \to \XX'$ be a morphism over $X$ with induced morphism $\phi:\Delta \to \Delta'$ on the bands.  Fix a character $\chi':\Delta' \to \Gm$, and let $\chi=\chi' \circ \phi$.  Then the pull-back functor $L\rho^*$ takes $\chi'$-homogeneous objects to $\chi$-homogeneous objects and induces an equivalence
    $$L\rho^*:\text{D}_{\text{qc},\chi'}(\XX') \to \Dqch(\XX)$$
    of triangulated categories.
\end{prop}

\begin{cor} \label{A c : pullback}
    Let $S$ be an algebraic stack, and let $\XX$ and $\YY$ be algebraic stacks over $S$.  Let $\Delta$ be a diagonalizable group, let $\pi_\XX:\XX \to X$ and $\pi_\YY:\YY \to Y$ be $\Delta$-gerbes, and let $\chi$ be a character.  Then the restriction of $Lp^*_{\XX}:\Dq(\XX)\to \Dq(\XX \times_S \YY)$ to $\Dqch(\XX)$ factors through $\text{D}_{\text{qc},(\chi,0)}(\XX \times_S \YY)$.
\end{cor}
\begin{proof}
    Apply \Cref{A p: decomp} and \Cref{A p: bands} to the cartesian square
    $$\begin{tikzcd}
{\XX\times_S \YY} \arrow[r, "\pi_\XX \times \text{id}"] \arrow[d, swap, "\text{id}\times \pi_\YY"] & {X \times_S \YY} \arrow[d, "\text{id}\times \pi_\YY"] \\
{\XX \times_S Y} \arrow[r, "\pi_\XX \times \text{id}"] \arrow[d, swap, "\text{pr}_\XX"] & {X \times_S Y} \arrow[d, "\text{pr}_{X}"] \\
{\XX} \arrow[r, "\pi_\XX"]           & {X}          
\end{tikzcd}.$$
\end{proof}

\begin{lem} \label{A l: duality 1}
   Let $\pi:\XX \to X$ be a $\Delta$-gerbe.   Then, $L\pi^*$ is left and right adjoint to $R\pi_*$.  In particular, if $X$ is locally noetherian, then the restriction of $\pi^\times=L\pi^*$ to $\Dbc(X)$ factors through $\Dbc(\XX)$.
\end{lem}
\begin{proof}
    It is clear that $L\pi^*$ is left adjoint to $R\pi_*$.  To prove that it's also right adjoint, we first establish some notation.  Let $i_0:\text{D}_{\text{qc},0}(\XX) \to \Dq(\XX)$ and $p_0:\Dq(\XX)\to\text{D}_{\text{qc},0}(\XX)$ denote the inclusion and projection via the decomposition of \Cref{A p: splitting}.  Since $L\pi^*$ factors through $\D_{\text{qc},0}(\XX)$, we have a natural isomorphism $i_0p_0L\pi^* \simeq L\pi^*$.  For any $\FF\in \Dq(\XX)$ and any $\GG \in \Dq(\XX)$, we have isomorphisms
$$\Hom_{\Dq(\XX)}(\FF,R\pi_*i_0p_0\GG) \simeq \Hom_{\text{D}_{\text{qc},0}(\XX)}(p_0L\pi^*\FF,p_0\GG)$$
$$\simeq \Hom_{\Dq(\XX)}(L\pi^*\FF,i_0p_0\GG)\simeq \Hom_{\Dq(\XX)}(L\pi^*\FF,\GG),$$
where the first isomorphism follow from the adjunctions $L\pi^* \vdash R\pi_*$ and $p_0 \dashv i_0$, the second follows from applying the fully faithful functor $i_0$ and using the isomorphism $i_0p_0L\pi^*\simeq L\pi^*$, and the third follows from noting that $L\pi^*$ factors through $\text{D}_{\text{qc},0}(\XX)$ and using the decomposition in \Cref{A p: splitting}.  The formation of the above isomorphisms are natural in $\FF$ and $\GG$, so $R\pi_*i_0p_0$ is right adjoint to $\pi^*$ and is therefore isomorphic to $R\pi_*$.  Thus we have natural isomorphisms
$$\Hom_{\Dq(\XX)}(R\pi_*\FF,\GG)\simeq \Hom_{\Dq(\XX)}(R\pi_*i_0p_0\FF,\GG)\simeq \Hom_{\text{D}_{\text{qc},0}(\XX)}(p_0\FF,p_0L\pi^*\GG)$$
$$\simeq \Hom_{\Dq(\XX)}(\FF,i_0p_0L\pi^*\GG)\simeq \Hom_{\Dq(\XX)}(\FF,L\pi^*\GG),$$
so that $L\pi^*$ is right adjoint to $R\pi_*$.  It's clear that $L\pi^*$ preserves coherence.  $\pi$ is flat, so $L\pi^*$ preserves boundedness.
\end{proof}

\begin{cor}\label{A c: dualizing complex} Let $f:X \to S$ be a concentrated morphism of algebraic stacks, and let $\pi:\XX \to X$ be a $\Delta$-gerbe.  Then the dualizing complex of $\XX \to S$ is the pullback of $f^\times\mathcal{O}_S$.    
\end{cor}
\begin{proof}
    Immediate.
\end{proof}

\begin{prop}\label{A p: duality}
    Let $\Delta$ be a diagonalizable group, let $X$ be a noetherian and concentrated algebraic stack, and let $\pi:\XX \to X$ be a $\Delta$-gerbe.  Let $f:Y \to X$ be a proper and tame morphism of algebraic stacks of finite tor-dimension, and let $g:\YY \to \XX$ denote the pullback of $f$ along $\pi$.  Then the restriction of $g^\times:\Dq(\XX) \to \Dq(\YY)$ to $\Dbc(\XX)$ factors through $\Dbc(\YY)$, and for any character $\chi$ of $\Delta$, the restriction of $g^\times$ to $\Dbcch(\XX)$ factors through $\Dbcch(\YY)$.
\end{prop}
\begin{proof}
        The first part is a follows from \cite[Theorem 3.1(ii)]{HP24}.  The second part follows from \Cref{A p: decomp}.
    \end{proof}
\begin{cor} \label{A c: duality}
    Let $S$ be a noetherian, concentrated algebraic stack, let $X$ be a noetherian algebraic stack which is concentrated over $S$, and let $Y$ be an algebraic stack which is proper, tame and of finite tor-dimension over $S$.  Let $\Delta$ be a diagonalizable group, let $\XX \to X$ and $\YY \to Y$ be $\Delta$-gerbes, and let $\chi$ be a character.  If we denote by $p_\XX:\XX \times_S \YY \to \XX$ the projection, then the restriction of $p_\XX^\times:\Dq(\XX) \to \Dq(\XX\times_S\YY)$ to $\Dbcch(\XX)$ factors through $\text{D}^b_{\text{coh},(\chi,0)}(\XX\times_S\YY)$.  Moreover, $p_{\XX}^\times(-)\simeq (p_{\XX}^\times\mathcal{O}_\XX)\otimes Lp_{\XX}^*(-)$.
\end{cor}
\begin{proof}
     We have a cartesian square
$$\begin{tikzcd}
{\XX\times_S \YY} \arrow[r, "\pi_\XX \times \text{id}"] \arrow[d, swap, "\text{id}\times \pi_\YY"] & {X \times_S \YY} \arrow[d, "\text{id}\times \pi_\YY"] \\
{\XX \times_S Y} \arrow[r, "\pi_\XX \times \text{id}"] \arrow[d, swap, "\text{pr}_\XX"] & {X \times_S Y} \arrow[d, "\text{pr}_{X}"] \\
{\XX} \arrow[r, "\pi_\XX"]           & {X}          
\end{tikzcd}$$
The result follows from functoriality, \Cref{A p: duality}, \Cref{A l: duality 1}, \Cref{A p: bands}, and \cite[Theorem 3.1(3)]{HP24}.
\end{proof}
\subsection{Twisted Fourier-Mukai Transforms}  Throughout this section, we fix a diagonalizable group $\Delta$, a cartesian square of algebraic stacks
$$\begin{tikzcd}
              & {X \times_S Y} \arrow[ld, swap, "p"] \arrow[rd, "q"] &               \\
{X} \arrow[rd, swap, "f"] &                          & {Y} \arrow[ld, "g"] \\
              & {S}                       &              
\end{tikzcd},$$
$\Delta$-gerbes $\pi_\XX:\XX \to X$ and $\pi_\YY:\YY \to Y$, and characters $\chi$ and $\psi$ of $\Delta$.  Let $p_\XX:\XX \times_S\YY \to \XX$ and $q_\YY:\XX \times_S\YY \to \YY$ denote the projections. 

Recall that for any object $K \in  \Dq(\XX \times_S \YY)$, one can define the \textit{Fourier-Mukai transform} with kernel $K$
$$\Phi_{K}:=Rq_{*}(K\otimes^L Lp^*(-)):\Dq(\XX) \to \Dq(\YY).$$
If $h:Z \to S$ is another morphism, $\pi_\ZZ:\ZZ \to Z$ is a $\Delta$-gerbe, and $L \in \Dq(\YY\times_S \ZZ)$, let $\pi_{12}:\XX \times_S \YY \times_S \ZZ \to \XX \times_S \YY$, and similarly for $\pi_{13}$, $\pi_{23}$, be the projections.  Let $r:\YY \times_S \ZZ \to \YY$ and $s:\YY \times_S \ZZ \to \ZZ$ be the projections.  Then we have the following:
\begin{lem}\cite[Section 5]{HP24}
    If $r$ and $q$ are tor-independent and $f,g,h$ are concentrated, then $\Phi_{L} \circ \Phi_K$ is a Fourier-Mukai transform with kernel
    $$R\pi_{13,*}(L\pi_{23}^* L \otimes^L L\pi^*_{12}K).$$
\end{lem}

Fourier-Mukai transforms interact well with the decompositions of \Cref{A p: splitting}:
\begin{lem} \label{A l: twitsed fm}
    If $K \in \text{D}_{\qc, (\chi^{-1},\psi)}(\XX \times_S \YY)$, then the restriction of $\Phi_K$ to $\text{D}_{\qc,\chi^{-1}}(\XX)$ factors through $\text{D}_{\qc,\psi}(\YY)$.
\end{lem}
\begin{proof}
    Apply \Cref{A c : pullback}, \Cref{A l:tensor}, and \Cref{A c: push}
\end{proof}
(Twisted) Fourier-Mukai transforms also often have adjoints:
\begin{prop}\label{A p: adjoints}Let $K \in \Dq(\XX \times_S \YY)$. \begin{enumerate}
    \item Suppose that $Y$ is noetherian and concentrated.  If $f$ is proper and tame of finite tor-dimension, then $\Phi_K:\Dq(\XX) \to \Dq(\YY)$ admits a right adjoint $H$.  If $K$ is perfect, then 
    $$H\simeq \Phi_{K^\vee \otimes^L q_\YY^\times \mathcal{O}_\YY}:\Dq(\YY) \to \Dq(\XX).$$
    \item Assume that $X$ is noetherian and concentrated and $K$ is perfect.  If $g$ is proper, tame, and flat with geometrically Gorenstein fibers of dimension $d$, then 
    $$G:=\Phi_{K^\vee \otimes^L p_\XX^\times\mathcal{O}_\XX}:\Dq(\YY) \to \Dq(\XX)$$
    is left adjoint to $\Phi_K$.
    \end{enumerate}
\end{prop}
\begin{proof}
We begin with (1). It suffices to show that $q_\YY^\times(-) \simeq Lq_\YY^*(-)\otimes q_{\YY}^\times\mathcal{O}_\YY$.  We have a commuting triangle
$$
\begin{tikzcd}
{\XX \times_S \YY} \arrow[r, "\text{id} \times \pi_\YY"] \arrow[d, swap, "q_\YY"] & {X \times_S \YY} \arrow[ld, "\text{pr}_\YY"] \\
{\YY}                     &              
\end{tikzcd}$$
By \cite[Theorem 3.2(3)]{HP24} and our assumptions on $f$, $\text{pr}_\YY^\times$ takes the desired form.  By \Cref{A l: duality 1}, $(\text{id} \times \pi_\YY)$ does too, so we are done.

For (2), we begin by repeating the above argument and showing that $p_\XX^\times(-) \simeq Lp_\XX^*(-)\otimes p_\XX^\times\mathcal{O}_\XX$.  It remains to show that $p^\times_\XX\mathcal{O}_\XX$ is an invertible object of $\Dq(\XX \times_S \YY)$.  Forming a commuting triangle as above, we reduce to proving that $\text{pr}_{\XX}^\times \mathcal{O}_\XX \in \Dq(\XX \times_S Y)$ is invertible.  We finish by applying \cite[Corollary 3.2(2)]{HP24}.
\end{proof}
\begin{cor}\label{A c :twisted adjoints}
    Let $K \in \text{D}_{\text{qc},(\chi^{-1},\psi)}(\XX \times_S \YY)$.
    \begin{enumerate}
        \item Suppose that $Y$ is noetherian, that $f$ is proper and tame of finite tor-dimension, and that $K$ is perfect.  Let $H$ be as in \Cref{A p: adjoints}.  Then the restriction of $H$ to $\text{D}_{\text{qc},\psi}(\YY)$ factors through $\Dqch(\XX)$.
        \item Suppose that $X$ is noetherian and concentrated, that $g$ is proper, tame, and flat with geometrically Gorenstein fibers, and that $K$ is perfect.  Let $G$ be as in \Cref{A p: adjoints}.  Then the restrction of $G$ to $\text{D}_{\text{qc},\psi}(\YY)$ factors through $\Dqch(\XX)$.
        \end{enumerate}
\end{cor}
\begin{proof}
    By \Cref{A l:tensor}, the corresponding Fourier-Mukai kernels are $(\psi^{-1},\chi)$-twisted, so may apply \Cref{A l: twitsed fm}.
\end{proof}
Under certain reasonable hypotheses, coherence is also preserved under (twisted) Fourier-Mukai transforms.

\begin{defn}
    Let $f:W \to T$ be a flat and locally of finite presentation morphism between quasi-compact algebraic stacks.  We say that $P \in \Dq(W)$ is $T$-\textit{perfect} if it is pseudocoherent on $W$ and if $P\otimes_{\mathcal{O}_W}^L f^*M$ is bounded for any $M \in \Dbq(T)$.
\end{defn}

\begin{rem}
    Our notion of our relatively perfect differs from the one used in \cite{HP24} \cite{Stacks}; this notion is stronger and easier for us to work with.  We feel that it is beyond the scope of this appedix to compare them.  The analogous notion was defined for schemes in \cite{ALS23}, where it was shown that under certain circumstances the two notions coincide \cite[Corollary 4.2]{ALS23}.
\end{rem}

\begin{lem}\label{A l: perfection}Let $f:W \to T$ be a flat and locally of finite presentation morphism between quasi-compact algebraic stacks.
    \begin{enumerate}
        \item If $T$ is noetherian, then $P \in \Dq(W)$ is pseudocoherent if and only if $P \in \Dmc(W)$.
        \item if $T=\Spec(k)$ is the spectrum of a field and $W$ is quasi-compact, then $P \in \Dq(W)$ is $T$-perfect if and only if $P \in \Dbc(W)$.
        \item if $P \in \Dq(W)$ is perfect, then $P$ is $T$-perfect.
    \end{enumerate}
\end{lem}
\begin{proof}
    (1) is in \cite[Section 2]{HP24}.  (2) and (3) are obvious.
\end{proof}
\begin{lem}\label{A l: coherence}
    Suppose that $X$ and $Y$ are noetherian, $q$ is proper, flat, and tame $p$ is flat and of finite type.  If $K \in \Dq(\XX \times_S \YY)$ is $\XX$-perfect, then the restriction of $\Phi_K:\Dq(\XX) \to \Dq(\YY)$ to $\Dbc(\XX)$ factors through $\Dbc(\YY)$.  If, in addition, $K \in \text{D}_{\text{qc},(\chi^{-1},\psi)}$, then the restriction of $\Phi_K$ to $\Dbcch(\XX)$ factors through $\text{D}^b_{\text{coh},\psi}(\YY)$.
\end{lem}

\begin{proof}
    By \Cref{A l: twitsed fm}, it suffices to check that being a bounded complex of coherent sheaves is preserved under such a Fourier-Mukai transform.  The only nontrivial part is to check that $K\otimes^L p_\XX^*(-)$ preserves boundedness.  This follows from the fact that $K$ is $\XX$-perfect.
    \end{proof}

\begin{prop}\label{A p: coherence and adjoints}
    Let $K \in \Dq(\XX \times_S \YY)$ be perfect.  
    \begin{enumerate}
        \item  Suppose that $f$ is proper tame and flat.  Let $H:\Dq(\YY) \to \Dq(\XX)$ be the right adjoint to $\Phi_K$ as in \Cref{A p: adjoints}.  Then the restriction of $H$ to $\Dbc(\YY)$ factors through $\Dbc(\XX)$.
        \item Suppose that $g$ is proper, tame, and flat with geometrically Gorenstein fibers.  Let $G:\Dq(\YY ) \to \Dq(\XX)$ be the left adjoint to $\Phi_K$ as in \Cref{A p: adjoints}.  Then the restriction of $G$ to $\Dbc(\YY)$ factors through $\Dbcch(\XX)$. 
    \end{enumerate}
\end{prop}
\begin{proof}
We only prove (1); (2) is proven similarly.  The only nontrivial part is showing that tensoring with $q_\YY^\times\mathcal{O}_\YY$ sends $\Dbc(\YY)$ to $\Dbc(\XX\times_S \YY)$.  This follows from considering the commuting triangle
$$
\begin{tikzcd}
{\XX \times_S \YY} \arrow[r, "\text{id} \times \pi_\YY"] \arrow[d, swap, "q_\YY"] & {X \times_S \YY} \arrow[ld, "\text{pr}_\YY"] \\
{\YY}                     &              
\end{tikzcd}$$

and using functoriality, \cite[Theorem 3.1(2,3)]{HP24}, and \Cref{A l: duality 1}.
\end{proof}
\begin{cor}\label{A c: coherent twisted adjoints}
    Let $K \in \text{D}_{\text{qc},(\chi^{-1},\psi)}(\XX \times_S \YY)$ be perfect.
    \begin{enumerate}
        \item Suppose that $f$ is proper tame and flat.  Let $H:\Dq(\YY) \to \Dq(\XX)$ be the right adjoint to $\Phi_K$ as in \Cref{A p: adjoints}.  Then the restriction of $H$ to $\text{D}^b_{\text{coh},\psi}(\YY)$ factors through $\Dbcch(\XX)$.
        \item Suppose that $g$ is proper, tame, and flat with geometrically Gorenstein fibers.  Let $G:\Dq(\YY ) \to \Dq(\XX)$ be the left adjoint to $\Phi_K$ as in \Cref{A p: adjoints}.  Then the restriction of $G$ to $\text{D}^b_{\text{coh},\psi}(\YY)$ factors through $\Dbcch(\XX)$.
    \end{enumerate}
\end{cor}
\begin{proof}
This follows immediately from \Cref{A c :twisted adjoints} and \Cref{A p: coherence and adjoints}.   
\end{proof}
Finally, we are ready to start dealing with twisted derived equivalence criteria.  We begin with two lemmas that enable us to reduce to working over an algebraically closed field.
\begin{lem} \label{A l: base change}
    Suppose that
    \begin{enumerate}
        \item $S$ is noetherian and concentrated.
        \item $f$ and $g$ are proper, flat, and tame.
        \item $K \in \text{D}^b_{\text{coh},(-\chi,\psi)}(\XX \times_S\ \YY)$ is perfect.
    \end{enumerate}
If $\Phi_{K_{s}}:\text{D}^b_{\text{coh},\chi}(\XX_{s}) \to \text{D}^b_{\text{coh},\psi}(\YY_{s})$ is fully faithful (resp.~ an equivalence) for every closed point $s$ of $S$, then $\Phi_K:\Dbcch(\XX) \to \text{D}^b_{\text{coh},\psi}(\YY)$ is fully faithful (resp.~ an equivalence).
\end{lem}
\begin{proof}
    Argue as in \cite[Proposition 5.7]{HP24}.
\end{proof}

\begin{lem}\label{A l: base change for fields}
    Let $f,g$ and $K$ be as in \Cref{A l: base change}.  Suppose that $S=\text{Spec}(k)$, where $k$ is a field.  Then $\Phi_{\EE}$ is full faithful (resp.~ an equivalence) if and only if $\Phi_{\EE_{\overline{k}}}$ is.
\end{lem}
\begin{proof}
Let $\eta$ and $\epsilon$ denote the unit and counit of the adjunction $\Phi_{\EE} \dashv H$.  $\Phi_{\EE}$ is fully faithful (resp.~ an equivalence) if and only if for $F \in \text{D}^b_{\text{coh},\psi}(Y)$, $\eta(F)$ (resp.~ $\eta(F)$ and $\epsilon(F)$) are isomorphisms.  This is condition that the cone of $\eta(F)$ (resp.~ $\eta(F)$ and $\epsilon(F)$) be zero.  This can be checked after base change.
\end{proof}

We recall the following definition.
\begin{defn}
    Let $\mathcal{T}$ be a triangulated category. A collection $\Omega$ of objects of $\mathcal{T}$ is called a \textit{spanning class} if the following two conditions hold:
    \begin{enumerate}
        \item If $t \in \mathcal{T}$ and $\Hom(t,\omega[n])=0$ for all $\omega \in \Omega$ and $n \in \mathbb{Z}$, then $t\simeq 0$,
        \item If $t \in \mathcal{T}$ and $\Hom(\omega[n],t)=0$ for all $\omega \in \Omega$ and $n \in \mathbb{Z}$, then $t\simeq 0$.
    \end{enumerate}
\end{defn}
The following is a useful example of a spanning class
\begin{defn} \label{A d: twisted skyscraper sheaves}
    Let $X$ be an algebraic space locally of finite presentation over an algebraically closed field $k$, let $\pi:\XX \to X$ be a $\Delta$-gerbe, and let $x \in X(k)$ be a $k$-point.  Denote by $j_x:\XX_x \to \XX$ the inclusion of the fiber over $x$.  For each character $\chi$ of $\Delta$, let $k(x,\chi)$ denote the pushforward of the unique simple object of $\text{QCoh}_\chi(\XX_x)$.  We refer to $k(x,\chi)$ as the $\chi$\textit{-twisted skyscraper sheaf supported at} $x$.  If $\Delta=\Gm$ and $\chi=\text{id}$, we simply write $k(x)$.
\end{defn}

\begin{rem} Note that any trivialization $\XX_x \simeq B\Delta$ identifies the unique simple object of $\text{QCoh}_{\chi}(\XX)$ with the line bundle on $B\Delta$ corresponding to the character $\chi:\Delta \to \Gm$.    
\end{rem}

\begin{prop}
    Suppose that $X$ is an algebraic space of finite type over an algebraically closed field $k$, let $\pi:\XX \to X$ be a $\Delta$-gerbe, and let $\chi$ be a character of $\Delta$.  Then the set
    $$\Omega=\text{$\{k(x,\chi): x \in X$ is a closed point$\}$ }$$
    is a spanning class for $\Dbcch(\XX)$.
\end{prop}
\begin{proof}
    Apply \cite[Proposition 6.14]{AHR23}, \cite[Corollary 7.17]{HP24}, and \Cref{A p: splitting}.
\end{proof}
Finally, we prove our main theorem.
\begin{thm}\label{A t: criterion}
    Let $X$ be a smooth, proper algebraic spaces over an algebraically closed field $k$, and let $Y$ be a Gorenstein, proper algebraic space over $k$.  Let $\pi_\XX:\XX \to X$ and $\pi_\YY:\YY \to Y$ be $\Delta$-gerbes, and let $\chi$ and $\psi$ be characters for $\Delta$.  Suppose that $\Delta$ is smooth.  Let $\EE \in \text{D}^b_{\text{coh},(\chi^{-1},\psi)}(\XX \times \YY)$ be perfect, and let
    $$F=\Phi_\EE:\Dbcch(\XX) \to \text{D}^b_{\text{coh},\psi}(\YY).$$ $F$ is a fully faithful if and only if:
    \begin{enumerate}
        \item For each closed point $x \in X$, $F$ induces an isomorphism
        $$\Ext^i(k(x,\chi),k(x,\chi)) \simeq \Ext^i(F(k(x,\chi)),F(k(x,\chi)))$$
        for $i=0,1$
        \item For each pair of closed points $x_1,x_2$ of $X$, and each integer $i$,
        $$\Ext^i(F(k(x_1,\chi_1)),F(k(x_2,\chi_2)))=0$$
        unless $x_1=x_2$ and $0 \leq i \leq \text{dim}(X)$.
    \end{enumerate}
    If $Y$ is assumed to be smooth, then any $\EE \in \text{D}^b_{\text{coh},(\chi^{-1},\psi)}(\XX \times \YY)$ is perfect.
\end{thm}
\begin{proof}
    Follow the proof of \cite[Theorem A]{HP24}.  To apply \cite[Corollary 7.18]{HP24}, use the decomposition of \Cref{A p: splitting} to ignore all other generalized closed points of $\XX$.  To produce $\Spec(A)$ in \cite[Corollary 7.18]{HP24}, choose a surjective étale morphism from a scheme $\tilde{X} \to X$ trivializing $\XX$, a lift of $x$ to $\tilde{X}$, and let $A$ be the local ring of $\tilde{X}$ at $x$. 
\end{proof}

\end{document}